\documentclass[10pt]{amsart}
\usepackage{amsmath,amsxtra,amssymb,amsthm,amsfonts}
\usepackage{graphicx}
\usepackage{color}
\usepackage{enumitem}
\usepackage[english]{babel}
\usepackage[autostyle]{csquotes}
\usepackage{mathrsfs}
\usepackage{etoolbox}

\numberwithin{equation}{section}
\newtheorem{lemma}{Lemma}[section]
\newtheorem{prop}[lemma]{Proposition}
\newtheorem{theorem}[lemma]{Theorem}

\newtheorem{cor}[lemma]{Corollary}

\newtheorem{rem}[lemma]{Remark}

\newtheorem{defi}[lemma]{Definition}
\newtheorem{exam}[lemma]{Example}

\renewcommand{\for}{\begin{eqnarray*}}

\newcommand{\mel}{\end{eqnarray*}}

\newcommand{\kl}{\pl \le \pl}

\newcommand{\lel}{\pl = \pl}

\newcommand{\tr}{\mathrm{tr}}

\newcommand{\zz}{{\mathbb Z}}

\newcommand{\vertiii}[1]{{\left\vert\kern-0.25ex\left\vert\kern-0.25ex\left\vert #1 
		\right\vert\kern-0.25ex\right\vert\kern-0.25ex\right\vert}}

\newcommand{\pl}{\hspace{.1cm}}

\newcommand{\qd}{\end{proof}\vspace{0.5ex}}

\newcommand{\si}{\sigma}

\newcommand{\pf}{\begin{proof}}

\newcommand{\xspace}{\hbox{\kern-2.5pt}}
\newcommand{\xyspace}{\hbox{\kern-1.1pt}}

\definecolor{LightGray}{rgb}{0.94,0.94,0.94}
\definecolor{VeryLightBlue}{rgb}{0.9,0.9,1}
\definecolor{LightBlue}{rgb}{0.8,0.8,1}
\definecolor{DarkBlue}{rgb}{0,0,0.6}
\definecolor{LightGreen}{rgb}{0.88,1,0.88}
\definecolor{MidGreen}{rgb}{0.6,1,0.6}
\definecolor{DarkGreen}{rgb}{0,0.6,0}
\definecolor{DarkGrreen}{rgb}{0,0.8,0}
\definecolor{VeryLightYellow}{rgb}{1,1,0.9}
\definecolor{LightYellow}{rgb}{1,1,0.6}
\definecolor{MidYellow}{rgb}{1,1,0.5}
\definecolor{DarkYellow}{rgb}{0.8,1,0.3}
\definecolor{VeryLightRed}{rgb}{1,0.9,0.9}
\definecolor{LightRed}{rgb}{1,0.8,0.8}
\definecolor{DarkRed}{rgb}{0.8,0.2,0}
\definecolor{DarkRedb}{rgb}{0.6,0.2,0}
\definecolor{DarkLila}{rgb}{0.8,0,1}
\definecolor{Beige}{rgb}{0.96,0.96,0.86}
\definecolor{Gold}{rgb}{1.,0.84,0.}
\definecolor{Goldb}{rgb}{0.7,0.3,0.5}
\definecolor{MyYellow}{rgb}{1.,0.84,0.8}

\begin{document}
\title[Noncommutative versions of the AGM inequality]{ Noncommutative versions of the arithmetic-geometric mean inequality}

\author{WAFAA ALBAR}

\author{MARIUS JUNGE}

\author{MINGYU ZHAO}

\address{UNVERSITY OF ILLINOIS AT URBANA-CHAMPAIGN}
\email{walbar2@illinois.edu}

\address{UNVERSITY OF ILLINOIS AT URBANA-CHAMPAIGN}
\email{junge@math.uiuc.edu}
\thanks{Partially supported by DMS 1501103 and BigData 1447879}

\address{UNVERSITY OF ILLINOIS AT URBANA-CHAMPAIGN}
\email{mzhao16@illinois.edu}
\thanks{}


\maketitle

	\mbox{} \hspace{4mm} Recht and R\'{e} in \cite{recht2012beneath} introduced the noncommutative arithmetic geometric mean inequality (NC-AGM) for matrices with a constant depending on the degree $d$ and the dimension $m$. In this paper we prove AGM inequalities with a dimension-free constant for general operators. We also prove an order version of the AGM inequality under additional hypothesis. Moreover, we show that our AGM inequality almost holds for many examples of random matrices .

\section{Introduction}
 \mbox{} \hspace{4mm} Variations of the arithmetic-geometric mean (AGM) inequality have many applications in analysis and geometry. As pointed out by R\'{e} and Recht in \cite{recht2012beneath}, noncommutative versions of the AGM inequalities are relevant to machine learning. In particular, their proof, which employed the classical MacLaurin inequalities, led to improved convergence rate of the of the algorithms in machine learning. \vspace{0.2in}\\ 
Let us recall the famous MacLaurin inequalities for positive real numbers $x_1,...,x_n$ and the normalized $d$-th symmetric sums as 
\begin{equation*}
S_d=\binom{n}{d}^{-1} \sum_{\substack{\tau\in [n]\\ |\tau|=k}}\prod_{i\in \tau }x_i \pl.
\end{equation*}
where $ 1\le d\le n $. According to the MacLaurin inequalities, we have
\begin{equation*}
S_1\geq \sqrt[2]S_2\geq \sqrt[3]S_3\geq ...\geq \sqrt[n]S_n \pl .
\end{equation*}
In particular,	$S_1\geq\sqrt[n]S_n$ is the standard AGM inequality. For more details about the classical AGM inequality see \cite{hardy1952inequalities}. In this paper,  we will discuss noncommutative versions of MacLaurin's inequalities. Indeed, we will consider a generalized AGM inequality for the norm and the order. It may come as a surprise to the operator algebra community that these inequalities are motivated by problems in machine learning, stochastic gradient method (see Buttou \cite{bottou1998online} and the reference there is in \cite{recht2012beneath}), and randomized coordinates descent (see Nesterov \cite{nesterov2012efficiency}). This interesting connection and an overview of known results on this topic can be found in \cite{recht2011parallel} and \cite{recht2012beneath}. In fact, these methods contain an iteration procedure which can be performed with or without replacement samples. Recht and R\'{e}, in \cite{recht2011parallel}, study the performance of both. They show that the expected convergence rate without replacement is faster than that with replacement. They proved this result by using a particular AGM inequality.\vspace{0.2in}

In the effort to generalize the classical AGM inequality to the noncommutative setting, a standard but naive procedure in noncommutative analysis is to replace scalars by operators. Famous examples of this strategy are Cauchy-Schwarz type inequalities for $C^*$-modules, Khintchine, and martingale inequalities.(See e.g. LP\cite{lust1986inegalites}, LPP\cite{lust1991non}, PXu\cite{pisier2003non}, Narcisse\cite{randrianantoanina2002non}, J\cite{junge2002doob} , JXu1\cite{junge2003noncommutative}, JXu2\cite{junge2003noncommutative}. For a general survey see \cite{pisier2003non}.) Proving these noncommutative extensions often employs a combination of functional analytic and combinatorial methods. In fact, the key results of this paper heavily rely on Pisier's interpretation of Rota's M\"{o}bius formulae for partitions. \vspace{0.2in}

A NC-AGM inequality would ask whether 
\begin{align}\label{AGM0}
A_1\cdots A_n \pl \stackrel{?}{\le} \pl  (\frac{1}{n}\sum_{j=1}^n A_j)^n 
\end{align}
holds for positive operators $A_{1},...,A_{n}$ on a Hilbert space. (In this context we shall interpret $x\le y$ as requiring that $y-x$ is positive semi-definite.) However, for positive operators $A$ and $B$, the product $AB$ may not be positive or even self-adjoint, so the inequality (1.1) may not make sense. Inspired by  Recht and R\'{e}, we modify \eqref{AGM0} by replacing the left hand side with the average of all the products of the operators $A_{i}$, which turns out to be self-adjoint. Following the MacLaurin approach, we may now ask whether the AGM inequality holds on average, i.e. 
\begin{align}\label{AGM1}
\frac{1}{n!} \sum_{\si\in S_n} A_{\si(1)}\cdots A_{\si(n)} \pl \stackrel{?}{\le} \pl  (\frac{1}{n}\sum_{j=1}^n A_j)^n.
\end{align}
 Unfortunately, we can not prove \eqref{AGM1} in general. A milder version of \eqref{AGM1} is to ask for 
\begin{align}\label{AGM2}
\|\frac{1}{n!} \sum_{\si\in S_n} A_{\si(1)}\cdots A_{\si(n)} \| \pl \stackrel{?}{\le} \pl  \|(\frac{1}{n}\sum_{j=1}^n A_j)^n\| \pl, 
\end{align}
where $\|x\|=\|x\|_{B(H)}$ refers to the standard operator norm of bounded operators on a Hilbert space $H$. The inequality \eqref{AGM2} is a particular case of the noncommutative MacLaurin inequalities discussed in \cite{recht2011parallel}. Indeed, for fixed $d$ we may consider the following average product of noncommutative operators of length $d$: 

\[ P_d(A_1,...,A_n) \lel \frac{1}{n\cdots (n-d+1)} 
\sum_{1\le j_1,...,j_d\le n \mbox{ \scriptsize all different}} 
A_{j_1}\cdots A_{j_d} \pl. \] 
We refer to the example in \cite{recht2012beneath} for the fact that the symmetrization for the operators in the AGM inequality is required. In \cite{recht2011parallel}, R\'{e} and Recht posed the following question: Is it true that for positive bounded operators $A_1,...,A_n$ on a Hilbert space one has
	\begin{align} \label{NAGM}
	\|P_d(A_1,...,A_n)\|^{1/d} &\le  \|P_1(A_1,...,A_n)\|  \pl?\pl 
	\end{align}
They proved that \eqref{NAGM} holds when $A_{1},...,A_{n}$ are matrices that mutually commute. Moreover, they observed that for operators $A_1,...,A_n$ on an $m$-dimensional Hilbert space one has 
\begin{align}\label{mAGM}
\|P_d(A_1,...,A_n)\|_{B(\ell_2^m)}^{1/d} \kl m \pl \|P_1(A_1,...,A_n)\|  \pl .
\end{align} 
We will prove the AGM inequality for the norm with a constant independent of the dimension $m$. 
\begin{theorem}
	For operators $A_{1},...,A_{n} \geq 0$ on a Hilbert space $H$,
	\[ \|P_d(A_1,...,A_n)\|^{1/d} \kl d \pl \|P_1(A_1,...,A_n)\| .\] 
\end{theorem} 
Let us now consider the \enquote{order version} of the AGM inequality. Here we add the additional assumption $\sum A_{i}=n$. In order to illustrate the technique we use generally, it is good to start with $d=3$.

\begin{theorem}\label{thm1.3}
	Let $ n\geq 6$. If $A_{1},...,A_{n}$ are self-adjoint operators such that $\sum_{i} A_{i}=n $. Then
	$  P_3(A_1,...,A_n)^{1/3}\leq 1 $.
\end{theorem}
 For the proof we consider the mean-zero operators $a_{i}:=A_{i}-1$. Observe the operators $a_{i}$ are self-adjoint and $\sum\limits_{i=1}^{n}a_{i}=0$. It follows easily that
$$P_{3}(A_{1},...,A_{n})=1+\binom{3}{1}P_{1}(a_{1},...,a_{n})+\binom{3}{2}P_{2}(a_{1},...,a_{n})+\binom{3}{3}P_{3}(a_{1},...,a_{n}) . $$  Straightforward computations using $\sum a_{i}=0 $ reveal that \\
$P_{1}(a_{1},...,a_{n})=\frac{(n-1)!}{n!}\sum\limits_{i}a_{i}=0$\\
$P_{2}(a_{1},...,a_{n})=\frac{(n-2)!}{n!} \sum\limits_{i\neq j}a_{i}a_{j}= \frac{(n-2)!}{n!} \Big( (\sum\limits a_{i})^{2}-\sum\limits a_{i}^{2}\Big)= -\frac{(n-2)!}{n!}\sum\limits a_{i}^{2}$\\
\begin{align*}
& P_{3}(a_{1},...,a_{n})=\frac{(n-3)!}{n!} (\sum\limits_{i\neq j\neq k} a_{i}a_{j}a_{k} )\\ 
&={\frac{(n-3)!}{n!}} \Big( (\sum\limits_{i\neq j \neq k} a_i)^3-(\sum\limits_{i=j} a_i ^2)(\sum\limits_{k} a_k) -(\sum\limits_{i} a_i)(\sum\limits_{j=k} a_j ^2)-\sum\limits_{j} \sum\limits_{i=k \neq j} a_i a_j a_i +2(\sum\limits_{i=j=k} a_i ^3)\Big)\\
&=2{\frac{(n-3)!}{n!}}\Big(\sum\limits_{i}a_{i}^{3}\Big).
\end{align*}

This leads to the form	$P_{3}(A_{1},...,A_{n})=1-\frac{3}{n(n-1)}\sum a_{i}^{2}+\frac{2}{n(n-1)(n-2)}\sum a_{i}^{3}.$ Together with $\sum a_{i}^3\leq  \|a_{i}\| \sum a_{i}^2 \leq n \sum a_{i}^2 ,$ this yields
\begin{equation}\label{d3}
P_{3}(A_{1},...,A_{n})\leq 1-\frac{3}{n(n-1)}\sum a_{i}^{2}+\frac{2n}{n(n-1)(n-2)}\sum a_{i}^{2}.
\end{equation}

Since
$\frac{2n}{n(n-1)(n-2)}\leq\frac{3}{n(n-1)}$ holds for all $n\geq 6$, the right side of \eqref{d3} is at most 1 and we are done. A far-reaching generalization of this idea leads to the following result.
\begin{theorem} Fix $n$ and $d$. Suppose $A_{1},...,A_{n}$ and $a_{i}$ are defined as above, $\sum\limits_{i} A_{i}=n$ and 
\begin{enumerate}
\item[i)]  $P_1(A_1,...,A_n)=\frac{\sum_{i} A_{i}}{n}=1,$ \item[ii)]  $\|(\sum a_{j}^2)^{\frac{1}{2}}\|\leq\frac{n}{3d}.$
\end{enumerate}
	Then the AGM inequality holds in the order sense:
    $$P_d(A_1,...,A_n)  \leq P_1(A_1,...,A_n)^d =1.$$
\end{theorem}
Note that these techniques work efficiently when $d$ is very large.\vspace{0.2in}\\
This paper is organized as follows. 
In section 2, we review the analytic and combinatorial tools needed to prove Theorem 1.2, especially Pisier's interpretation of Rota's results on M\"{o}bius transforms for partitions. In section 3, we combine the results from section 2 with Pisier's group construction for partitions in \cite{pisier2000inequality} to obtain our key norm and order estimate. In section 4 and 5,  a combination of Pisier's partition method and probabilistic results allow \enquote{almost AGM} inequalities hold in many different scenarios. We confirm the AGM inequality up to $\varepsilon$ for many random matrices, in particular Wishart random matrices, more general vector-valued moments of convex bodies, and freely independent operators.
We should point out that in contrast to results on averages of random matrices in R\'{e} and Recht in \cite{recht2012beneath}, our estimates hold with high probabilities.

\section{Partition and  M{\"o}bius Formula }
 \mbox{} \hspace{4mm} We need some definitions from the combinatorial theory of partitions. Let $\mathbb{P}_d$ be the lattice of all the partitions of $\{1,...,d\}$. For two partitions $\sigma$ and $\pi$, we write $\sigma \leq \pi $ if every block of the partition $\sigma$ is contained in some block of $\pi$ (i.e., any block of the partition of $\pi$ can be written as a union of blocks of $\sigma$). In other words, $\pi$ is a refinement of $\sigma$. There are two trivial partitions, $\dot{0}$ and $\dot{1}$, where $\dot{0}$ is the partition into $n$ singletons and $\dot{1}$ is the partition of a single block. For a partition $\pi$, $\nu({\pi})$ is the number of the blocks of the partition $\pi$ and $r_{i}(\pi)$ is the number of blocks of $\pi$ with cardinality $i$ such that $\sum_{i=1}^{d}ir_i(\pi)=d;$ and  $\sum_{i=1}^{d}r_i(\pi)=\nu(\pi)$. For more information on partitions, see \cite{andrews1998theory} and \cite{rota1964foundations}.\vspace{0.2in} \\ 
 Let us recall some main results on the M\"{o}bius function $\mu$ in \cite{pisier2000inequality} which are crucial for our paper.

\begin{prop}\label{prop} 
	(Pisier's M{\"o}bius inversion formula)
	For any $d\in \mathbb{N}$ there exists a function $\mu :\mathbb{P}_{d} \times \mathbb{P}_{d}\longrightarrow \zz$ such that for every vector space $V$ and functions $\phi:\mathbb{P}_{d}\longrightarrow V$ and $ \psi:\mathbb{P}_{d}\longrightarrow V$, we have the following properties:
	\begin{enumerate}[label=\roman*)]
		\item If $\psi(\sigma)=\sum\limits_{\pi\leq\sigma}\phi(\pi)$, then $\phi(\sigma)=\sum\limits_{\pi\leq\sigma}\mu(\pi,\sigma)\psi(\pi)$;
		\item If $\psi(\sigma)=\sum\limits_{\pi\geq\sigma}\phi(\pi)$, then $\phi(\sigma)=\sum\limits_{\pi\geq\sigma}\mu(\sigma,\pi)\psi(\pi)$;
		 \item Moreover, $\forall\sigma\neq\dot{0}$, $\sum\limits_{\dot{0}\leq\pi\leq\sigma}\mu(\pi,\sigma)=0$.
	\end{enumerate}

\end{prop}
The next result provides precise formulas for the M\"{o}bius function $\mu$ in special cases.
\begin{theorem} \label{thm22}
	
	The M\"{o}bius function satisfies the following properties:
	\begin{enumerate}[label=\roman*)]
		\item $\mu(\dot{0},\dot{1})=(-1)^{d-1}(d-1)!$.
		\item$\mu(\dot{0},\pi)=\prod_{i=1}^{d}[(-1)^{i-1}(i-1)!]^{r_{i}(\pi)}$, and consequently,
		\item  $\sum_{\pi\in P_{d}}|\mu(\dot{0},\pi)|=d!  .$
	\end{enumerate}
\end{theorem}

 If $\sigma$ is a partition of $\{1,...,d\}$, then there exists a coordinate function $f:\{1,...,d\} \rightarrow \{1,...,\nu(\sigma)\}$ such that $f^{-1}(t)=A_t$ where each $A_{t}$ represents a block in our partition. Note that this coordinate function isn't unique. For every partition $\sigma$ we can fix an enumeration of the blocks $f:\{1,2,...,d\}\longrightarrow \{1,2,...,|\sigma|\}$ where $\sigma$:=$\langle j_{1},j_{2},...,j_{d}\rangle$.
This means $j_r=j_s$ if and only if ${r,s}\in A_{r,s}$ where $A_{r,s}$ is a block in $\sigma=\langle  j_{1},j_{2},...,j_{d}\rangle$. Using this notation we define the restricted and full partition for elements from an algebra.

\begin{defi}
	Let $\mathscr{A}$ be an algebra and $x_{j_{i}}^{i}\in \mathscr{A}$. The restricted partition is defined by: 
	$$\langle\sigma\rangle=\sum_{\langle j_1,j_2,...,j_d\rangle=\sigma}x_{j_1}^{1}...x_{j_d}^{d}.$$
	The full partition with elements $x_{j_i}$ is given by:
	$$[\sigma]=\sum_{\pi\geq\sigma}\langle\pi\rangle .$$ 
\end{defi}	
	
	The restricted and full partitions, which are denoted as $\langle\sigma\rangle$ and $[\sigma]$, respectively, give expressions for the elements in the given $B(H)$ according to the algebraic combinatorial partition $\sigma $. 
In order to understand the difference between the definition of restricted partition and full partition, consider the following example.
\begin{exam}
Let both the numbers of total samples and chosen samples be 3 ($n=d=3$). Then for the full partition $[1~2,3]$ , with the assumption that $x_{j}^{i}=x_{j}$ we have
\begin{eqnarray*}
 [1~2,3]&=&(\sum x_i ^2)(\sum x_i)\\
 &=&\langle1~2,3\rangle+\langle1~2~3\rangle\\
 &=&\sum_{i_{1}=i_{2}\neq i_{3}} x_{i_{1}}^{2}x_{i_{3}} + \sum_{i_{1}=i_{2}=i_{3}} x_{i_{1}} ^3.\\
 \end{eqnarray*}
 Whereas the restricted partition $\langle1~2,3\rangle$ is defined as
 $\langle1~2,3\rangle = \sum_{i_{1}=i_{2}\neq i_{3}} x_{i_{1}}^{2}x_{i_{3}}.$

\end{exam}	

	 We reformulate Pisier's  M\"{o}bius inversion formula in our context. 
	
	\begin{prop}\label{pr2.5}
	Let $x_{j}^{k}\in \mathscr{A}$ as above. Then we have	
    \begin{equation}
	\langle\pi \rangle=\sum_{\nu\geq\pi}\mu(\pi,\nu)[\nu],~where~ 	[\pi]=\sum_{\nu\geq\pi}\langle \nu \rangle, 
	\end{equation}

	\begin{equation}
	\langle\pi \rangle=\sum_{\nu\leq\pi}\mu(\pi,\nu)[\nu],~where~ 	[\pi]=\sum_{\nu\leq\pi}\langle \nu \rangle.
	\end{equation}

	Moreover, we have  
	\begin{equation} \label{eq2.3}
	\langle\dot{0} \rangle=[\dot{0}]+\sum_{\dot{0}\lneqq \nu \leq \dot{1}}\mu(\dot{0},\nu)[\nu].
	\end{equation}
	\end{prop}	



	
 In \cite{pisier2000inequality}, in order to separate different partition blocks into disjoint subspaces, Pisier uses a trick to embed operators $x_{ik}\in B(H)$ into $B(K\otimes H)$ (for another Hilbert space $K$). Our first goal is to modify Pisier's trick by using matrix units.\\
 
 Consider first the trivial partition that has only one block $[1,2,\cdots,d]$. We can write 
 \begin{eqnarray*}
 	\dot{1}=[1,2,\cdots,d]&=&\sum x_{i_1}^1x_{i_2}^2\cdots x_{i_d}^d \\
 	&= &(\sum e_{1i_1} \otimes x_{i_1}^1  )\times (\sum e_{i_2i_2} \otimes x_{i_2}^2 )\times \cdots\\
 	&\times& (\sum e_{i_{d-1}i_{d-1}} \otimes  x_{i_{d-1}}^{d-1})\times (\sum e_{i_d1} \otimes  x_{i_d}^d ).
 \end{eqnarray*}
 
 Now if we have 6 elements and our partition $\sigma$ has two crossing blocks, one containing $\{1,3,4,6\}$ and the other containing $\{2,5\}$, then the full partition of $\sigma$ will be of the form:
 \begin{equation*}
 [\sigma]= \sum_{\substack{i_1=i_3=i_4=i_6\\ i_2=i_5}} x_{i_1}x_{i_2}x_{i_3}x_{i_4}x_{i_5}x_{i_6}.
 \end{equation*}
 We rewrite these elements into a tensor form, as follows:
 \begin{eqnarray*}
 	Z_{i_1}& =& e_{1i_1}\otimes 1 \otimes x_{i_1}~,~~~~~~~~Z_{i_2} = 1 \otimes  e_{1i_2} \otimes x_{i_2} \\
 	Z_{i_3}& =& e_{i_3i_3}\otimes 1 \otimes x_{i_3},~~~~~~~~Z_{i_4} = e_{i_4i_4}\otimes 1 \otimes x_{i_4} \\
 	Z_{i_5}& =& 1\otimes e_{i_51} \otimes x_{i_5}~,  ~~~~~~~~Z_{i_6} = e_{i_6 1}\otimes 1 \otimes x_{i_6} .  \\
 \end{eqnarray*}
 With this new notation, we get 
 \begin{eqnarray*}
 	[\sigma]&=& \sum_{\substack{i_1=i_3=i_4=i_6\\ i_2=i_5}}x_{i_1}x_{i_2}x_{i_3}x_{i_4}x_{i_5}x_{i_6} \\
 	&=& \sum_{i_1,i_2,i_3,i_4,i_5,i_6}Z_{i_1}Z_{i_2}Z_{i_3}Z_{i_4}Z_{i_5}Z_{i_6} \\
 	&=& \prod_{j=1}^{6}  (\sum_{i_j} Z_{i_j}) = \prod_{j=1}^{6} Z_j , ~~(Z_j :=\sum_{i_j} Z_{i_j}  ).\\\
 \end{eqnarray*} 
 
 In a more general setting, assume $\sigma$ has more than one block. Denote $A_{1}$,\ldots,$A_{|\sigma|}$ as the blocks of the partition $\sigma$ with cardinality larger than one. 
 
 Then we define
 \begin{equation}\label{eq1}
 Z_{j_k}^k \in B(H)^{\otimes^{|\sigma|} }\otimes B(H)
 \end{equation}
 as follows:\\
 \begin{align*}
 \forall  k \in A_{1} ,~& Z_{j_k}^k= t_{A_1}(j_k) \otimes 1 \otimes\cdots \otimes x_{j_k}^k\\
 \forall  k \in A_{2} ,~& Z_{j_k}^k=1\otimes t_{A_2}(j_k)\otimes1 \cdots\otimes x_{j_k}^k\\
\forall   k \in  A_{|\sigma|} ,~& Z_{j_k}^k= 1\otimes\cdots \otimes t_{A_{|\sigma|}}(j_k) \otimes x_{j_k}^k ,
 \end{align*}
 where 
 \begin{displaymath}
 t_{A_m(j_k)}=\left\{
 \begin{array}{cc}
 e_{1j_k} & j_k=\min A_m\\
 e_{j_kj_k} &  \text{otherwise}\\
 e_{j_k1} & j_k=\max A_m .\\
 \end{array} \right.
 \end{displaymath}
 Here, $\min A_m$ means the smallest index number and $\max A_m$ means the largest index number in the partition $A_m$.
 Finally, if $k$ belongs to singleton block of the partition $\sigma$, then we set
 \begin{equation*}
 Z_{j_k}^k=1\otimes \cdots \otimes 1\otimes x_{j_k}^k.
 \end{equation*}
 To sum up, the method places each element  into larger spaces, which will allow us to interchange the summation and multiplication as in the above example  and the following lemma.
 \begin{lemma}\label{1}
For an arbitrary partition $\sigma$  for d elements, we have 
$$[\sigma]=\sum\limits_{i_{1},...,i_{d}}Z_{i_{1}}^{1}...Z_{i_{d}}^{d}.$$
\end{lemma} 
Indeed, this immediately follows from 
	\raggedbottom
 \begin{equation*}
 Z_{i_j}^i \cdot Z_{i_k}^k =0, \text{if }  i_j\neq i_k.
 \end{equation*} 

 Follow Pisier's result in \cite{pisier2000inequality}; we deduce the following norm estimate. 
 
 \begin{theorem}\label{lem32}
 	For an arbitrary partition $\sigma$  for d elements, we have
 	$$\| [\sigma] \|_{B(H)}  \leq \prod\limits_{k=1 }^d \Big (\| \sum\limits_{j_k} Z_{j_k}^k\| \cdot 1_{ \sigma_s}(k) + \| \sum\limits_{j_k} Z_{j_k}^k  \|\cdot 1_{  \sigma_{ns} }(k) \Big). $$ 
    Moreover, 
 	$\| [\sigma] \|_{B(H)}  \leq \prod_{k \in \sigma_{s} }  \| \sum\limits_{j_k} Z_{j_k} \|  \times \prod_{k\in \sigma_{ns}}  \||(Z_{j_{k}})|\|,$\\
 	where $\||(Z_{j_{k}})|\|=\max \{\|\sum Z_{j_{k_1}}Z_{j_{k_1}}^*\|^{\frac{1}{2}},~ \|\sum Z_{j_{k_p}}^*Z_{j_{k_p}}\|^{\frac{1}{2}}, \sup_{j_k} \|Z_{j_k}\|\}.$
 
 	Here $\sigma_{s} $ means the set of singletons in the partition $\sigma $, and $\sigma_{ns} $ means the set of non-singleton elements in the partition $\sigma $. The functions $1_{ \sigma_{ns} }(k),  1_{\sigma_{s}}(k)$ represent the characteristic functions, i.e.
 	\begin{displaymath}
 	1_{ \sigma_{ns}  }(k)=\left\{
 	\begin{array}{cc}
 	1 & k \in \sigma_{ns}\\
 	0 &  \text{otherwise}\\
 	\end{array} \right. ,~~
 	1_{ \sigma_{s}  }(k)=\left\{
 	\begin{array}{cc}
 	1 & k \in \sigma_s\\
 	0 &  \text{otherwise}\\
 	\end{array} \right.
 	\end{displaymath}
 	
 \end{theorem}
 \begin{proof} Taking the norm for the full partition, we have 
 	\begin{align}\nonumber
 		&\|[\sigma]\|=\| \sum\limits_{\pi \geq \sigma} \langle\pi\rangle \|= \| \sum\limits_{ \langle  j_{1},\cdots,j_{d}\rangle  \geq \sigma  } x_{j_{1}}^1\cdots x_{j_{d}}^d\| \\\label{en3}
 		&=\|\sum \limits_{j_1,j_2,\ldots,j_d} Z_{j_{1}}^1\cdots Z_{j_{d}}^d\| \\\label{en5}
 		&=\|\prod_{k \in \sigma_{s} }   \sum\limits_{j_k} Z_{j_k}^k  \cdot \prod_{k \in \sigma_{ns} }  \sum\limits_{j_k} Z_{j_k}^k  \| \\\nonumber
 		&\leq \|  \prod_{k \in \sigma_{s} }   \sum\limits_{j_k} Z_{j_k}^k \|  \cdot \| \prod_{k \in \sigma_{ns} }  \sum\limits_{j_k} Z_{j_k}^k \|  \\\nonumber
 		&\leq \prod_{k \in \sigma_{s} }  \| \sum\limits_{j_k} Z_{j_k}^k \|  \cdot \prod_{k \in \sigma_{ns} } \| \sum\limits_{j_k} Z_{j_k}^k  \|. \nonumber
 	\end{align}
 	The equality \eqref{en3} comes from Lemma \ref{1}. 
 	The equality \eqref{en5} follows from the definition of $Z_{j_k}^k$, which means it allows us to perform summation first and then multiplication.\vspace{0.1in}
 	 Next, 
\raggedbottom
 	\begin{align*}
 	& \|[\sigma]\|_{B(H)}\leq \prod_{k \in \sigma_{s} }  \| \sum\limits_{j_k} Z_{j_k} \|  \cdot \prod_{k \in \sigma_{ns} } \| \sum\limits_{j_k} Z_{j_k}  \|\\
 	& \leq  \prod_{k \in \sigma_{s} }  \| \sum\limits_{j_k} Z_{j_k} \| \times \prod_{k\in A_m \subset \sigma_{ns}    }     \| \sum\limits_{j_k} Z_{j_k}  \|\cdot (1_{\min A_m} + 1_{\max A_m } + 1_{\text{mid}  ~A_m })\\
 	&\leq \prod_{k \in \sigma_{s} }  \| \sum\limits_{j_k} Z_{j_k} \|  \times  \\
 	&  \prod_{k\in A_m \subset \sigma_{ns}    }    \Big(  \| \sum\limits_{j_k}  Z_{j_k}  \|\cdot 1_{\min A_m} +  \| \sum\limits_{j_k} Z_{j_k}  \|\cdot 1_{\max A_m } + \| \sum\limits_{j_k} Z_{j_k}  \|\cdot 1_{\text{mid} ~A_m }\Big) \\
    \end{align*}
   \begin{align*}
 	&\leq \prod_{k \in \sigma_{s} }  \| \sum\limits_{j_k} Z_{j_k} \| \times   \\
 	&  \prod_{k\in A_m \subset \sigma_{ns}    }    \Big(  \| \sum\limits_{j_k}  Z_{j_k} Z_{j_k}^* \|^{\frac{1}{2}}\cdot 1_{\min A_m} +  \| \sum\limits_{j_k} Z_{j_k}^*Z_{j_k}  \|^{\frac{1}{2}}\cdot 1_{\max A_m } + \sup_{j_k} \| Z_{j_k}  \|\cdot 1_{\text{mid} ~A_m }\Big) \\
 	&\leq \prod_{k \in \sigma_{s} }  \| \sum\limits_{j_k} Z_{j_k} \|  \times \prod_{k\in \sigma_{ns}}  \||(Z_{j_{k}})|\|, 
 	\end{align*}	
 	where
 	$\||(Z_{j_{k}})|\|=\max \{\|\sum Z_{j_{k_1}}Z_{j_{k_1}}^*\|^{\frac{1}{2}},~ \|\sum Z_{j_{k_p}}^*Z_{j_{k_p}}\|^{\frac{1}{2}}, \sup_{j_k} \|Z_{j_k}\|\}.$
 \end{proof}
 
  The next corollary states the norm estimate in $B(H)$ rather than in $B(K \otimes H)$. For simplicity we replace $x_{i_{k}}^{k}$ by $x_{i_{k}}$.
  
 \begin{cor} \label{th3.3}
 	If $\sigma $ is a partition and $x_{j_k}$ is a self-adjoint operator for arbitrary $k\in \{1,...,d\}$, then 
 	$$\| [\sigma] \|_{B(H)} \leq \prod_{k \in \sigma_{s} }  \| \sum x_{j_k}   \|\cdot \prod_{k \in \sigma_{ns} } \| \sum x_{j_k}^2   \|^{\frac{1}{2}}$$

 \end{cor}
 \begin{proof}
 
 	We need to discuss two cases:\\
 	(i) For $ k  \in \sigma_{s}$,
 	$\| \sum\limits_{j} Z_{j_k} \| 
 	=\| \sum 1\otimes \cdots \otimes x_{j_k} \|=\| 1\otimes \cdots \otimes \sum x_{j_k}\| 
 	=\|\sum x_{j_k}\|. $\\
 	(ii) For $A_m \in \sigma_{ns}$,
 	\begin{align}
 	\|\sum Z_{j_{k_1}}Z_{j_{k_1}}^*\|^{\frac{1}{2}}\nonumber
 	&=\|\sum [1\otimes \cdots \otimes e_{1j_{k_1}}\otimes \cdots \otimes x_{j_{k_1}}  ]  \cdot [1\otimes \cdots \otimes e_{j_{k_1}1} \otimes \cdots \otimes x_{j_{k_1}}^*  ]     \| ^{\frac{1}{2}} \\\nonumber
 	&=\| \sum 1\otimes \cdots \otimes e_{11} \otimes \cdots \otimes x_{j_{k_1}}x_{j_{k_1}}^*\| ^{\frac{1}{2}}\\
 	&=\| 1\otimes \cdots \otimes \sum x_{j_{k_1}}x_{j_{k_1}}^* \|^{\frac{1}{2}}=\|\sum x_{j_{k_1}}x_{j_{k_1}}^* \|^{\frac{1}{2}}=\|\sum x_{j_{k_1}}^2\|^{\frac{1}{2}}.\label{in3.3}
 	\end{align}
 	and
 	\begin{align}
 	\|\sum Z_{j_{k_p}}^*Z_{j_{k_p}}\|^{\frac{1}{2}} \nonumber
 	&=\|\sum [1\otimes \cdots \otimes e_{1j_{k_p}}\otimes \cdots \otimes x_{j_{k_p}}^*  ]  \cdot [1\otimes \cdots \otimes e_{j_{k_p}1} \otimes \cdots \otimes x_{j_{k_p}}  ]     \|^{\frac{1}{2}}  \\\nonumber
 	&=\| \sum 1\otimes \cdots \otimes e_{11} \otimes \cdots \otimes x_{j_{k_p}}^*x_{j_{k_p}}\|^{\frac{1}{2}}
 	=\| 1\otimes \cdots \otimes \sum x_{j_{k_p}}^*x_{j_{k_p}} \| ^{\frac{1}{2}}\\
 	&=\|\sum x_{j_{k_p}}^*x_{j_{k_p}} \|^{\frac{1}{2}}=\|\sum x_{j_{k_p}}^2\|^{\frac{1}{2}}. \\\nonumber
 	\end{align}
 	For the middle term, we have
 	\begin{align}\nonumber
 	\sup \limits_{k\in \{k_{2} ,\ldots,k_{p-1} \} } \sup\limits_{j_k} \|Z_{j_k}\|&=\sup \limits_{k\in \{k_{2} ,\ldots,k_{p-1} \} }  \sup\limits_{j_k} \|Z_{j_k}^*Z_{j_k}\|^{\frac{1}{2}}=\sup \limits_{k\in \{k_{2} ,\ldots,k_{p-1} \} } \sup\limits_{j_k} \|x_{j_k}^*x_{j_k}\|^{\frac{1}{2}} \\
 	&\leq \sup \limits_{k\in A_m } \| \sum x_{j_k}^2\|^{\frac{1}{2}}.
 	\end{align} 
 	Combining (i) and (ii) finishes the proof.
 \end{proof}

\section{AGM inequality for the norm and for the order}
 In this section we prove the AGM inequality for the norm and for the order. We need the following lemma which handles positive or self-adjoint operators $\{x_{i_{k}}\}$ in a C*-algebra $\mathscr{A}.$
\begin{lemma} \label{lem34}

	\begin{enumerate}
			\item[(i)]If $x_{j_k} \geq 0$, then $\|\sum x_{j_k}^{2}\|^{\frac{1}{2}} \leq \|\sum x_{j_k}\|.$\\
			\item[(ii)]  If $x_{j_k}$ are self-adjoint, then $\|\sum x_{j_k}^{2}\|^{\frac{1}{2}}=\|(\sum x_{j_k}^{2} )^{\frac{1}{2}}\| .$
	\end{enumerate}
	\end{lemma}

\begin{proof} 
	(i) Indeed, we have
	\begin{eqnarray*}
		\|\sum x_{j_k}^{2}\|^{\frac{1}{2}}&=&\|\sum x_{j_k}^{\frac{1}{2}} x_{j_k}x_{j_k}^{\frac{1}{2}}\|^{\frac{1}{2}} \\
		&\leq& (\|\sum x_{j_k}\|^{\frac{1}{2}}\cdot\|\sum x_{j_k}\| \cdot \|\sum x_{j_k}\|^{\frac{1}{2}})^{\frac{1}{2}} \\
		&=&\|\sum x_{j_k}\|.
	\end{eqnarray*}
	
		(ii) Holds trivially using $\|x^2\|=\|x\|^2$, for $x=(\sum x_{j_k}^{2} )^{\frac{1}{2}}$. \qedhere\\
\end{proof}
\subsection{AGM inequality for the norm}
Now we have done all the preparation to prove the NC-AGM inequality for the norm.
\begin{theorem}\label{w1}
	Suppose $x_1, \dots, x_n$ are positive operators in $B(H)$. Then
	\begin{equation}
	\|P_d(x_1,...,x_n)\|_{B(H)}^{1/d} \leq d~\|P_1(x_1,...,x_n)\|_{B(H)}  .
	\end{equation} 
\end{theorem}
\begin{proof}
	From Corollary \ref{th3.3} and Lemma \ref{lem34}, we deduce that for a given  arbitrary partition $\sigma$ and positive elements $x_{j_{k}}=x_{j}$, we have 

	\begin{equation*}
	\|[\sigma]\|_{B(H)} \leq \|\sum x_j\|^{d}.
	\end{equation*}
	Recall identity \ref{eq2.3} from Proposition \ref{pr2.5}:
	\begin{equation}\label{eqm1}
	\langle 1,\cdots,d \rangle=[1,\cdots,d]+ \sum\limits_{\upsilon \gneqq \dot{0}} \mu(\dot{0},\nu)[\nu],~~\text{where}  \sum\limits_{\upsilon \gneqq \dot{0}} | \mu(\dot{0},\nu)|=d!-1.
	\end{equation}
	Taking the norm of both sides of the equality \eqref{eqm1} we get
	\begin{eqnarray*}
		\|\langle 1,\cdots,d \rangle\|_{B(H)} &=&\|[1,\cdots,d]+ \sum\limits_{\upsilon \gneqq \dot{0}} \mu(\dot{0},\nu)[\nu]\|_{B(H)}  \\
		&\leq& \|[1,\cdots,d]\|_{B(H)} +\sum\limits_{\upsilon \gneqq \dot{0}} |\mu(\dot{0},\nu)|\|[\nu]\|_{B(H)}   \\
		&\leq&\|\sum x_j\|_{B(H)}^d+(d!-1)\|\sum x_j\|_{B(H)}^d \\
		&\leq& d!\|\sum x_j\|_{B(H)}^d\\
		&=& d!n^d \|\frac{1}{n}\sum x_j\|_{B(H)}^d \\
		&=& d!n^d \|P_{1}(x_{1},...,x_{n})\|_{B(H)}.
	\end{eqnarray*} 
	Thus,
	\begin{equation*}
	\|P_{d}(x_{1},...,x_{n})\|_{B(H)}\leq \frac{d!n^d(n-d)!}{n!} \|P_{1}(x_{1},...,x_{n})\|_{B(H)}.
	\end{equation*}
	
	Denote $C(n,d):=\frac{d!n^d(n-d)!}{n!}$, and for fixed $d$ define $f(n):=\sum\limits_{i=0}^{d-1} \log \frac{n}{n-i} $. Then 
	\begin{eqnarray*}
		C(n,d)&=&\frac{d!n^d(n-d)!}{n!}=\frac{d!n^d}{n(n-1)(n-2)\cdots(n-d+1)} \\
		&=& d!\cdot\frac{n}{n}\cdot\frac{n}{n-1}\cdot\frac{n}{n-2} \cdots\frac{n}{n-d+1} \\
		&=& d!\cdot \exp(f(n)).
	\end{eqnarray*}
	
	Since $f(n)$ is a decreasing function in $n$, $C(n,d)$ is also a decreasing function with respect to the variable $n$. From the definition of $d$, we know $n\geq d$, so $\max\limits_{n\geq d}~C(n,d)= C(d,d)=d^d.$ \qedhere
\end{proof}


	

\subsection{AGM inequality for the order}
Recall that the average product is defined by:
$$P_{d}(x_{1},x_{2},...,x_{n})=\frac{(n-d)!}{n!}\sum\limits_{\langle \sigma \rangle=\dot{0}}x_{i_{1}}...x_{i_{d}}.$$
\begin{lemma}\label{lemma3.7}
	Let $\{x_i\}$  be a finite family of positive operators in $B(H)$ which satisfy the condition $\sum\limits_{i=1}^n x_i= n$.
	If $a_{i}:=x_{i}- 1$ then
	\begin{equation} 
	P_{d}(x_{1},x_{2},...,x_{n})=1+\sum\limits_{k=1}^{d}\binom{d}{k}P_{k}(a_{1},a_{2},...,a_{n}).\label{inq3.4}
	\end{equation}
\end{lemma}
\begin{proof}
	This lemma can be proved by two methods. The first method is by induction which is left to the reader. For the convenience of the reader we give the second proof, using the binomial identity. Then we have
	\begin{eqnarray*}
		P_{d}(x_{1},...,x_{n})&=&\frac{(n-d)!}{n!}\sum_{\langle \sigma \rangle=\dot{0}}x_{i_{1}}...x_{i_{d}}\\
		&=&\frac{(n-d)!}{n!}\sum_{\langle \sigma \rangle=\dot{0}}(a_{i_{1}}+1)(a_{i_{2}}+1)...(a_{i_{d}}+1)=1+\sum_{k=1}^{d}\lambda_{k}P_{k}(a_{1},...,a_{n}).
	\end{eqnarray*}
	
	Let $x_{1}=x_{2}=....=x_{n}=t$, where $t=a+1$. Then
	$$P_{d}(x_{1},...,x_{n})=t^{d}=(1+a)^{d}=1+\sum_{k=1}^{d}\binom{d}{k}a^{k},$$
	which implies that $\lambda_{k}=\binom{d}{k}$, so 
	$P_{d}(x_{1},...,x_{n})=1+\sum_{k=1}^{d}\binom{d}{k}P_{k}(a_{1},...,a_{n}).\qedhere$ 
\end{proof}

In Theorem \ref{thm1.3} for $d$=3, we deduce that each term in $P_{3}(x_{1},...,x_{n})$ has an upper bound of some scalar multiple of $\sum a_{i}^{2}$. For $d>3$, we need the following lemma.

\begin{lemma}\label{lemma3.8}
	If $\{x_{i}\},\{a_{i}\}$ are defined as above,
	 then $$\max_{i}\| a_i\|  \leq \| \sum a_i^2 \|^{\frac{1}{2}} \leq \|\sum_{i}x_{i}^2\|^{\frac{1}{2}}. $$ In particular, $\| a_i\|^k \leq n^k  \|\frac{1}{n^2}\sum_{i}x_{i}^2\|^{\frac{k}{2}} .$
	\end{lemma}

\begin{proof}
	
	Since we have $a_j^2 \leq \sum a_i	^2$, 
	$$\|a_i\| =\|a_i^2\|^\frac{1}{2}\leq  \|\sum a_i	^2\|^\frac{1}{2}.$$
	Moreover, for each $a_i$, we have $x_i=a_i+1$. Thus $$\sum x_i^2=\sum a_i^2+n\geq \sum a_{i}^2$$ 
	This finishes the proof.\qedhere
\end{proof}

Note that for a partition with $d=3$, the proof of the AGM inequality in the order sense was easily done in the introduction. However, the proof is much more complicated for $d\geq 4$. The complication comes from crossing partitions, so we need the following useful known lemma \cite{paulsen2002completely}.
\begin{lemma} \label{Co} Assume $a,b\in B(H)$ and $t\geq 0$. Then
	
	\begin{eqnarray*} 
	 &(1)& -(a^*a+b^*b) \leq a^*b+b^*a \leq a^*a+b^*b \\ 
	 &(2)& ab+b^{*}a^{*}\leq t^{2}aa^{*}+t^{-2}b^{*}b \\
	\end{eqnarray*}
\end{lemma}
 To prove (1), we start by observing $(a+b)^*(a+b), (a-b)^*(a-b) \geq 0$. This directly gives $-(a^*a+b^*b)\leq a^*b+b^*a$ and $ a^*b+b^*a \leq a^*a+b^*b$. It is clear that (2) is a special case of (1), using the assumptions that $a=ta^*$ and $b=t^{-1}b$ for the upper bound of (1). \vspace{0.2in}

The two previous lemmas will help in establishing our result for general case of the AGM inequality for the order. 
For convenience, we will write $A_{i}:=\sum_{i} Z_{i}$ where $Z_{i}$ is defined as at the beginning of Section 3.1. We now provide upper and lower bounds for $P_{d}(a_{i_1},...,a_{i_n}).$ 
\begin{lemma} If $\{a_i\}$ and $\{x_{i}\}$ are defined as above, then for $S=\|\sum x_{i}^{2}\|^{1/2}$
	$$   -\frac {(n-d)!} {n!} d!~S^{d-2}~\sum a_i^2  \leq  P_d(a_1,a_2,\cdots,a_n) \leq \frac {(n-d)!} {n!} d!~ S^{d-2}\sum a_i^2.$$
\end{lemma}
\begin{proof}
	From Proposition \ref{pr2.5}, we know
	$$\frac{n!}{(n-d)!}P_d(a_1,a_2,\cdots,a_n)=\langle\dot{0} \rangle_d =[ \dot{0}]_d + \sum_{\dot{0}\lneqq \nu \leq \dot{1}}\mu(\dot{0},\nu)[\nu]_d.$$
	We will prove first the case when $\mu(\dot{0},\nu) \geq 0$. We will obtain an upper bound for the sum $ [\nu]_d$ by introducing $[\bar{\nu}]_d$ as the following: 
	$$	[\nu]_d =\sum\limits_{\langle\dot{i_1,i_2,\cdots,i_d} \rangle \geq \nu} a_{i_1}a_{i_2}\cdots a_{i_d},$$
	$$[\bar{\nu}]_d := \sum\limits_{\langle\dot{i_1,i_2,\cdots,i_d} \rangle \geq \nu} a_{i_d}a_{i_{d-1}}\cdots a_{i_1}.$$
	Here the $\bar{\nu}$ can be viewed as the transposition of the partition $\nu$. By Theorem \ref{thm22}, we have 
	 $\mu(\dot{0},\pi)=\prod_{i=1}^{d}[(-1)^{i-1}(i-1)!]^{r_{i}(\pi)}$. So $\mu(\dot{0},\nu)=\mu(\dot{0},\bar{\nu})$. Thus, we can sum these two items together.\\
	\textbf{Claim}: For every partition $\nu$ and $S=\|\sum x_{i}^{2}\|^{1/2}$ we have 
	\begin{equation}\label{transpose}
	-2~S^{d-2}\sum a_i^2\leq [\nu]_d + [\bar{\nu}_d]\leq 2~S^{d-2}\sum a_i^2.
	\end{equation}
	The idea here is to use our modification of Pisier's trick for these two partitions. Recall that $Z_{i_1}=e_{1i_{1}}\otimes a_{i_{1}}$ is for the first component in the partition,
	$Z_{i_j}=e_{jj}\otimes a_{i_{j}}$ is for the elements in the middle of the partition, and $Z_{i_d}=e_{i_{d}1}\otimes a_{i_d}$  is for the last element in the partition. Then we have
	\raggedbottom
\begin{eqnarray*}
		[\nu]_d+[\bar{\nu}]_d &=&\sum\limits_{\langle\dot{i_1,i_2,\cdots,i_d} \rangle \geq \nu} a_{i_1}a_{i_2}\cdots a_{i_d} +a_{i_d}a_{i_{d-1}}\cdots a_{i_1}  \\
		&=& \sum_{i_1}Z_{i_1}...\sum_{i_d}Z_{i_d}+\sum_{i_d}Z_{i_d}^{*}...\sum_{i_1} Z_{i_1}^{*}\\
		&=& A_{1}...A_{d}+A_{d}^{*}...A_{1}^{*}.
	\end{eqnarray*}
	By applying Lemma \ref{Co} with $S=\|\sum x_{i}^{2}\|^{1/2}$,
	\begin{eqnarray}
		[\nu]_d+[\bar{\nu}]_d
		&=& A_{1}\cdots A_{d}+A_{d}^{*}\cdots A_{1}^{*}\\\label{inq3.1}\nonumber
		&\leq& t^2A_{1}A_{1}^{*}+ t^{-2}A_{d}^{*} \cdots A_{2}^{*}A_{2} \cdots A_{d}\\\nonumber
		&\leq& t^2A_{1}A_{1}^{*}+t^{-2}\prod_{j=2}^{d-1}\|A_{j}^{*}A_{j}\| A_{d}^{*}A_{d}\\\nonumber
		&\leq& t^2 A_{1}A_{1}^{*}+t^{-2} \prod_{j=2}^{d-1}\|A_{j}\|^{2} A_{d}^{*}A_{d}\\\label{in3.6}
		&\leq& t^2 A_{1}A_{1}^{*}+t^{-2}\prod_{j=2}^{d-1}\|\sum a_{j}^{2}\|A_{d}^{*}A_{d}  \\\label{in3.7}
		&\leq& t^{2}A_{1}A_{1}^{*}+t^{-2} \|\sum a_{i}^{2}\|^{d-2}A_{d}^{*}A_{d}\\
		&\leq& \|\sum a_{i}^{2}\|^{d/2-1}(A_{1}A_{1}^{*}+A_{d}^{*}A_{d}) \nonumber
		\leq 2\sum a_{i}^{2}~S^{d-2}.\nonumber
	\end{eqnarray}
	 Indeed, if our partition contains the singleton then $[\nu]_d+[\bar{\nu}]_d$ is already zero. Hence we may assume there are no singletons in our partition as it also can be noticed in inequality \eqref{in3.7}. Indeed, if the index is a singleton in partition $\nu$, then it is controlled by the summation norm $\|\sum a_i\|$ which is zero by our construction. On the other hand, if the index is in a non-singleton block, then by Theorem \ref{th3.3}, it is controlled by the square norm $\|\sum a_i^2\|.$ Therefore, in both cases, $\|A_i\|$ is controlled by the square norm of $a_{i}$. To get inequality \eqref{in3.6}, we may apply the norm equality as in equality \eqref{in3.3} from section 2. For the inequality \eqref{in3.7}, we use Lemma \eqref{Co} by choosing $t^2=S^{d/2-1}$.
	 Then we have 
		\begin{align*}
		&\frac{n!}{(n-d)!}P_d(a_1,a_2,\cdots,a_n)
		=[\dot{0}]_d + \sum_{\dot{0}\lneqq \nu \leq \dot{1}}\mu(\dot{0},\nu)[\nu]_d =\sum_{\dot{0}\lneqq \nu \leq \dot{1}}\mu(\dot{0},\nu)[\nu]_d \\
		&=\sum_{\mu(\dot{0},\nu)\geq 0}\mu(\dot{0},\nu)[\nu]_d + \sum_{\mu(\dot{0},\nu)\leq 0}\mu(\dot{0},\nu)[\nu]_d \\
		&= \frac{1}{2} \Big( \sum_{\mu(\dot{0},\nu)\geq 0}\mu(\dot{0},\nu)[\nu]_d +\sum_{\mu(\dot{0},\bar{\nu})\geq 0}\mu(\dot{0},\bar{\nu})[\bar{\nu}]_d \Big)\\
		&+ \frac{1}{2} \Big(\sum_{\mu(\dot{0},\nu)\leq 0}\mu(\dot{0},\nu)[\nu]_d+\sum_{\mu(\dot{0},\bar{\nu})\leq 0}\mu(\dot{0},\bar{\nu})[\bar{\nu}]_d \Big)
		\end{align*}
		\begin{align*}
		\frac{n!}{(n-d)!}P_d(a_1,a_2,\cdots,a_n)&\leq \sum_{\mu(\dot{0},\nu)\geq 0}\mu(\dot{0},\nu)S^{d-2} \sum a_i^2 -\sum_{\mu(\dot{0},\nu)\leq 0}\mu(\dot{0},\nu)S^{d-2} \sum a_i^2\\
		&=  \sum |\mu(\dot{0},\nu)|~S^{d-2} (\sum a_i^2)=d!~S^{d-2}\sum a_i^2.
	\end{align*}
	For the lower bound, the proof is similar to the one above replacing $A_{1}$ by $-A_{1}$. 
\end{proof}
\begin{theorem} \label{thm1.11}
	(AGM inequality for the order) Fix $n$ and $d$. Let $x_{1},...,x_{n}$ be self-adjoint operators such that $\sum\limits_{i}x_{i}=n$ and $a_{i}=x_{i}-1$ as above. Assume the following conditions hold:
	\begin{enumerate}
		\item[i)]  $P_1(x_1,...,x_n)=\frac{\sum_{i} x_{i}}{n}=1,$
		\item[ii)]  $\|(\sum x_{i}^2)^{\frac{1}{2}}\|\leq\ \frac{n}{3d}.$ 
	\end{enumerate}
	Then the AGM inequality holds in the order sense:
	 $$P_d(x_1,x_2,\cdots,x_n) \leq \Big(\frac{\sum_{i}x_{i}}{n}\Big)^d=1.$$
\end{theorem}

\begin{proof} According to Lemma \ref{lemma3.8}, we have $\|\sum a_{i}^{2}\|^{1/2}\leq \frac{n}{3d}$. Using this upper bound for the average of noncommutative operators $a_{i}$ with the identity \eqref{inq3.4} where $S=\|\sum x_{i}^{2}\|^{1/2}\leq \Delta n$ and let $\Delta:=\frac{1}{3d}$, we have
	\begin{eqnarray*}
		P_{d}(x_{1},x_{2},...,x_{n})&=& 1+\sum_{k=1}^{d}\binom{d}{k}P_{k}(a_{1},a_{2},...,a_{n})\\
		&=& 1-\binom{d}{2}\frac{(n-2)!}{n!}(\sum a_i^2)+\sum_{k=3}^{d}\binom{d}{k}P_k(a_{1},a_{2},...,a_{n})  \\
		&\leq&  1-\binom{d}{2}\frac{(n-2)!}{n!}(\sum a_i^2)+ \sum_{k=3}^{d}\binom{d}{k}\frac{(n-k)!}{n!}k! \Delta^{k-2}n^{k-2}(\sum a_i^2) .\\
	\end{eqnarray*}
Now we need the following condition:
	\begin{equation} \label{eq3.5}
	\binom{d}{2}\frac{(n-2)!}{n!} \pl \stackrel{?}{\geq} \pl\sum_{k=3}^{d}\binom{d}{k}\frac{(n-k)!}{n!}k! \Delta^{k-2}n^{k-2} .	
	\end{equation}
	Simplifying the right hand side gives
	\begin{eqnarray}\nonumber
		\sum_{k=3}^{d}\binom{d}{k}\frac{(n-k)!}{n!}k! \Delta^{k-2}n^{k-2} &=& \sum_{k=3}^{d} \frac{d!}{(d-k)!k!}\frac{(n-k)!k!}{n!}\Delta^{k-2}n^{k-2} \\\nonumber
		&=& \sum_{k=3}^{d} \frac{d!}{(d-k)!}\frac{(n-k)!}{n!}\Delta^{k-2}n^{k-2} \\\label{a}
		&=& \frac{1}{n(n-1)}\sum_{k=3}^{d} \frac{d!}{(d-k)!}\frac{n^{k-2}}{(n-2)\cdots(n-k+1)}\Delta^{k-2}.\\\nonumber
	\end{eqnarray}
	Fix $k$, and denote $f(n):=\frac{n^{k-2}}{(n-2)\cdots(n-k+1)}$. Then, by taking the logarithm, we have $g(n):=\log f(n)=\sum_{i=2}^{k-1} \log\frac{n}{n-i}$. Observe that $g(n)$ is a decreasing function and thus $f(n)$ is  a decreasing function as well. Therefore, we get the inequality:
	\begin{equation}\label{b}
	\frac{n^{k-2}}{(n-2)\cdots(n-k+1)} \leq \frac{d^{k-2}}{(d-2)\cdots(d-k+1)}.
	\end{equation}
    We continue the calculation in \eqref{a} with the help of inequality \eqref{b}, we have 
	\begin{eqnarray*}
		\sum_{k=3}^{d}\binom{d}{k}\frac{(n-k)!}{n!}k! \Delta^{k-2}n^{k-2} &=& \frac{1}{n(n-1)}\sum_{k=3}^{d} \frac{d!}{(d-k)!}\frac{n^{k-2}}{(n-2)\cdots(n-k+1)}\Delta^{k-2}\\
		&\leq&\frac{1}{n(n-1)}\sum_{k=3}^{d} \frac{d!}{(d-k)!}\frac{d^{k-2}}{(d-2)\cdots(d-k+1)}\Delta^{k-2}\\
		&\leq& \frac{1}{n(n-1)}\sum_{k=3}^{d}d(d-1)d^{k-2}\Delta^{k-2} \\
		&=& \frac{d(d-1)}{n(n-1)}\frac{d\Delta(1-(d\Delta)^{d-2})}{1-d\Delta} 
		\leq \frac{d(d-1)}{n(n-1)}\frac{d\Delta }{1-d\Delta}. 
	\end{eqnarray*}
	With our choice of $\Delta=\frac{1}{3d}$ we deduce indeed $ \frac{d(d-1)}{n(n-1)}\frac{d\Delta }{1-d\Delta}  \leq \binom{d}{2}\frac{(n-2)!}{n!} $ and this completes the proof.\qedhere
\end{proof}
\section{AGM inequality for random matrices}
In this section, we prove a version of the NC-AGM inequality for random matrices. We start with a deviation inequality. Let us use the norm $\vertiii{X}_p=(E\|X\|_{B(H)}^{p})^{1/p}$ defined for a random variable $X:\Omega \rightarrow B(H)$.

\begin{prop} \label{lm4.1}
	Let $\{a_{i}\}$ be a family of self-adjoint random operators. Let $\varepsilon > 0$, $p\geq 2$, $p_d=\frac{p}{d}$ and $x_{i}=a_{i}+1$. Define  
	\begin{enumerate}
	 \item[(i)] $\varepsilon_{p}:=\vertiii{\frac{1}{n}\sum{a_{i}}-E\frac{1}{n}\sum{a_{i}}}_p$ ,\\
	 \item[(ii)] $\delta_{p}:=\frac{1}{n}\vertiii{(\sum a_{i}^2)^{1/2}}_p$ ,\\
	 \item[(iii)] $\gamma_{p}: =\max (\varepsilon_{p},\delta_{p})$.
	\end{enumerate}
	 Assume $\sum\limits_{i} Ea_{i}=0$ and $\gamma_{p} \leq \frac{1}{3d}$ and $\varepsilon=3d\gamma_{p}$ \\
	Then, $\vertiii{P_{d}(x_{1},...,x_{n})-EP_{d}(x_{1},...,x_{n})}_{p_d}\leq \varepsilon.$
\end{prop} 
\begin{proof}
	From the assumption above, we get that
	\raggedbottom
	 $$\vertiii{(\frac{1}{n}\sum a_{i})}_p=\varepsilon_{p}.$$ Fix a partition $\nu$. According to Theorem \eqref{lem32} and by using H\"older's inequality we have that
	\begin{align*}
	E\|[\nu]\|^{p_d}_{\infty}&\leq E\Bigg(  \|(\sum a_{i}^{2})^{1/2}\|_{\infty}^{(d-|\nu_{s}|)p_d}  \|(\sum a_{i})\|_{\infty}^{|\nu_{s}|p_d}  \Bigg)\\
	&=  E\Bigg(  \|(\sum a_{i}^{2})^{1/2}\|_{\infty}^{ \frac{(d-|\nu_{s}|)p_dd}{d}  } \cdot  \|(\sum a_{i})\|_{\infty}^{ \frac{|\nu_s|p_dd}{d}}    \Bigg)\\
	\end{align*}
	
	\begin{align*}
	&\leq  \Bigg( E(\|(\sum a_{i}^{2})^{1/2}\|^{p_dd }_{\infty}   \Bigg)^{\frac{p_d(d-|\nu_{s}|)}{p_dd}}  \Bigg( E \|(\sum a_{i})\|_{\infty}^{p_dd} \Bigg)^{\frac{p_d|\nu_{s}|}{p_dd}} \\
	&= \Bigg( E(\|(\sum a_{i}^{2})^{1/2}\|^{p}_{\infty}   \Bigg)^{\frac{p_{d}(d-|\nu_{s}|)}{ p}}  \Bigg( E \|(\sum a_{i})\|_{\infty}^{p} \Bigg)^{\frac{p_d|\nu_{s}|}{ p}} \\
	&= \Bigg(\vertiii {\sum a_{i}^{2})^{1/2} }_{p}  \Bigg) ^{p_d(d-|\nu_{s}|)}  \Bigg(  \vertiii{\sum a_{i}} _{p}   \Bigg)^{p_d|\nu_{s}|} \\\vspace{0.75in}
	&= (\delta_{p} \cdot n)^{p_d(d-|\nu_{s}|)}.(\varepsilon_{p} \cdot n)^{p_d|\nu_{s}|}\\
	&= \delta_{p} ^{p_d(d-|\nu_{s}|)}\varepsilon_{p}^{p_d|\nu_{s}|}n^{p} = \delta_{p} ^{p_d(d-|\nu_{s}|)}\varepsilon_{p}^{p_d|\nu_{s}|}n^{p_{d}d}.
	\end{align*} 
	Since $\gamma_{p}=\max(\delta_{p},\varepsilon_{p})$,
	\begin{equation}\label{lalala}
	\vertiii{[\nu]}_{p_d}=(E\|[\nu]\|^{p_d}_{\infty})^{\frac{1}{p_d}}\leq \gamma_{p_dd}^{d}\cdot n^{d} 
	\end{equation}
	 By using our definition of $\gamma_{p}$ and the upper bound for inequality \eqref{lalala} we obtain
	\begin{align}\nonumber
	&(E\|P_{k}(a_{1},...,a_{n})-EP_{k}(a_{1},...,a_{n})\|_{\infty}^{p_d})^{1/p_d} \\\nonumber
	&\leq \frac{(n-k)!}{n!}\sum|\mu(0,\nu)|\Big(E(\|[\nu]-E[\nu]\|_{\infty})^{p_d} \Big)^{1/p_d}\\\nonumber
	&\leq \frac{(n-k)!}{n!}\sum|\mu(0,\nu)|\cdot 2(E\|[\nu]\|_{\infty}^{p_d})^{1/p_d}\\
	&\leq 2\frac{(n-k)!}{n!} k!\gamma_{p_d k}^{k}n^{k}.
	\end{align}
	From the above we will have 
	\begin{align}\nonumber
	&\vertiii{P_{d}(x_{1},...,x_{n})-EP_{d}(x_{1},...,x_{n})}_{p_d}\\\nonumber
	&=\vertiii{\sum_{k=1}^{d}{{d}\choose{k}}(P_{k}(a_{1},...,a_{n})-EP_{k}(a_{1},...,a_{n}))}_{p_d}\\\nonumber
	&=(E\|\sum_{k=1}^{d}{{d}\choose{k}}(P_{k}(a_{1},...,a_{n})-EP_{k}(a_{1},...,a_{n}))\|_{\infty}^{p_d})^{1/p_d}\\\nonumber 
	&\leq \sum_{k=1}^{d}{{d}\choose{k}} (E\|(P_{k}(a_{1},...,a_{n})-EP_{k}(a_{1},...,a_{n}))\|_{\infty}^{p_d})^{1/p_d}\\\nonumber
	&\leq 2\sum_{k=1}^{d}{{d}\choose{k}} \frac{(n-k)!}{n!} k!\gamma_{p_d k}^{k}\cdot n^{k}
	=2\sum_{k=1}^{d}\frac{d!}{k!(d-k)!}\frac{(n-k)!}{n!} k!\gamma_{p_d k}^{k}\cdot n^{k}\\
	&\leq 2\sum_{k=1}^{d}\frac{d!(n-k)!n^{k}}{(d-k)!n!}
	\gamma_{p}^{k}.  \label{inq4}
	\end{align}
	Recall the definition  $\gamma_{p_dk}=\max(\delta_{p_dk},\varepsilon_{p_dk})$. Each $\delta_{p_dk},\varepsilon_{p_dk}$ is increasing since $L_{p_dk}$ is defined as probability space which is norm increasing in probability measure. Thus $\gamma_{p_dk} \leq \gamma_{p_dd}=\gamma_{p}, \forall~ 1\leq k \leq d$, which justifies the last inequality \eqref{inq4}.
	Let $f(n)=\frac{d!(n-k)!n^{k}}{(d-k)!n!}$. This function is a decreasing function in $n$, so $f(d)=\max f(n)=d^{k}.$ Then we have 
	\begin{align*}
	&\vertiii{P_{d}(x_{1},...,x_{n})-EP_{d}(x_{1},...,x_{n})}_{p_d} \\&\leq 2\sum_{k=1}^{d}(d\cdot \gamma_{p})^{k}
	=2\cdot d\cdot \gamma_{p}\frac{(1-(d\cdot \gamma_{p})^{d})}{1-d\cdot \gamma_{p}} 
	\leq 2\cdot \frac{d\cdot \gamma_{p}}{1-d\cdot \gamma_{p}}\leq \varepsilon. 
	\end{align*}
	The last inequality follows from $d\cdot \gamma_{p} \leq \frac{\varepsilon}{1- \varepsilon/2}$.
\end{proof}
	We now present conditions for positive random operators $\{x_{i}\}$ where $a_{i}=x_{i}-1$. Note that for $A:=\sum_{i=1}^{n}\frac{a_i}{n} $, we have$$E\|A- EA\|_{p} =  E    \| (\frac{\sum x_{j}}{n})-E(\frac{\sum x_{j}}{n}) \|_{p}.$$Therefore, whenever we control the ${x_{i}}$'s, we control the ${a_{i}}$'s. 
	
\begin{lemma}\label{Lem4.3}
	Let $\{x_{i}\}$ be a family of self-adjoint random operators. 
	Then $$ \vertiii{(\sum (x_i -1)^2)^{1/2}}_p \leq 6 \vertiii{(\sum x_i^2)^{1/2}}_{p} $$
	
\end{lemma}
\begin{proof}
	Observe that $\vertiii{(\sum x_i^2)^{1/2}}_p= \vertiii{\sum x_i \otimes e_{i,1}}_{p} $ is given by the column norm. 
	Define operators $\phi :C_n(B(H))\rightarrow C_n(B(H))$ and $\Phi: C_{n}\rightarrow C_{n}$ such that $\Phi(\alpha_{i})=(\frac{1}{n}\sum\limits_{i}\alpha_{i})_{j}$ where $\phi=\Phi\otimes Id$. Then it is easy to check that $\|\Phi\|_{cb}=\|\phi\|_{cb}\leq 1$. Indeed 
	$$\|\sum\limits_{j=1}^{n} e_{j,1}\otimes \phi(y_{i})_{j}\|= \|\sum\limits_{j=1}^{n} e_{j,1} \otimes (\frac{1}{n} \sum_{i=1}^n y_i)\| 
	= \frac{1}{\sqrt{n}}\|\sum_{j=1}^{n} y_j\|  \leq \|\sum\limits_{i} e_{1,i}\otimes y_{i}\|.$$
	Denote $z_i := x_i - Ex_i$, so $\| \sum\limits_{i} e_{1,i}\otimes (Id+ \Phi )(z_i)\| \leq 2 \|\sum\limits_{i} e_{1,i}\otimes x_i\|$. Also,
	\begin{align*}
	(Id+ \phi )(z_i) &=x_i - Ex_i + \frac{1}{n} \sum x_i - \frac{1}{n} \sum E x_i =x_i -1 - Ex_i + \frac{1}{n} \sum x_i,\\
	(x_i -1)&=(Id+ \phi )(z_i) + Ex_i -\frac{1}{n} \sum x_i.
	\end{align*}
	By triangle inequality, we can get
	\begin{align*}
	&\vertiii{\sum (x_i -1)\otimes e_{i,1}} \\
	&\leq \vertiii{\sum (Id+ \phi )(z_i) \otimes e_{i,1}}+ \vertiii{\sum Ex_i \otimes e_{i,1}}+ \vertiii{\sum_j (\frac{1}{n} \sum x_i)  \otimes e_{j,1}} \\
	& \leq 2 \vertiii{ \sum z_i \otimes e_{i,1}} +\vertiii{\sum Ex_i \otimes e_{i,1}} +  \vertiii{\sum_j (\frac{1}{n} \sum x_i) \otimes e_{j,1}}\\
	& \leq 2 \vertiii{ \sum (x_i - E x_i) \otimes e_{i,1}} +\vertiii{\sum Ex_i \otimes e_{i,1}} +  \vertiii{\sum_j (\frac{1}{n} \sum x_i) \otimes e_{j,1}} \\
	& \leq 2 \vertiii{\sum x_i \otimes e_{i,1}}+3\vertiii{\sum Ex_i \otimes e_{i,1}}   + \vertiii{\frac{1}{n} \sum x_i}\\
	& \leq 2 \vertiii{(\sum x_i^2)^{\frac{1}{2}}}+ 3\vertiii{\sum x_i \otimes e_{i,1}}+\vertiii{(\sum x_i^2)^{\frac{1}{2}}}  \leq 6 \vertiii{(\sum x_i^2)^{\frac{1}{2}}}.  
	\end{align*}
	The second-to-last inequality $\vertiii{\sum Ex_i \otimes e_{i,1}} \leq \vertiii{\sum x_i \otimes e_{i,1}}$ follows from the fact that conditional expectation from $E: L_{\infty} (\Omega, B(H)) \rightarrow B(H)$ is a complete contraction. The inequality  $$\vertiii{\frac{1}{n} \sum x_i} \leq \vertiii{(\sum x_i^2)^{\frac{1}{2}}}$$ is true by the Cauchy-Schwarz inequality.\qedhere
\end{proof}

Thanks to Theorem \ref{lm4.1} and Lemma \ref{Lem4.3} we obtain the following deviation result.

\begin{theorem} \label{th4.3}
	Let $p\geq 2$, $p_d:=\frac{p}{d}$, and $\{x_{i}\}$ be a random family of positive operators such that $Ex_{i}=1$. Define 
	\begin{enumerate}
	\item [(i)] $\varepsilon_{p}:=\vertiii{\frac{1}{n}\sum{x_{i}}-E\frac{1}{n}\sum{x_{i}}}_p ,$\\
	\item [(ii)] $\delta_{p}:=\frac{1}{n}\vertiii{(\sum x_{i}^2)^{1/2}}_p$ ,\\
	\item [(iii)] $\gamma_{p}: =\max (\varepsilon_{p},4\delta_{p})$.
	\end{enumerate}
    If $3d\cdot \gamma_{p} \leq 1$ 
	then $$\vertiii{P_{d}(x_{1},...,x_{n})-EP_{d}(x_{1},...,x_{n})}_{p_d}\leq 3d\cdot \gamma_{p}.$$
\end{theorem} 

\begin{cor} \label{col2.4} 
	 If in addition $\{x_{i}\}$ are matrix-valued i.i.d. Then 
	$$\vertiii{P_{d}(x_{1},...,x_{n})}_{p_d}\leq 1+3d\cdot \gamma_{p}.$$
\end{cor}
\begin{proof}
		Since ${x_{i}}$'s are matrix-valued i.i.d, then $
		E(P_{d}(x_{1},...,x_{n}))=P_{d}(Ex_{1},...,Ex_{n}).$ Moreover, for $\varepsilon:=3d\cdot \gamma_{p}$
	by (ii) in the above Theorem \ref{th4.3}, we have
	$$\| (\sum E(x_{i})^{2})^{1/2} \|=\|\sum E(x_i) \otimes e_{i,1}  \|_{C_n \otimes B(H)} \leq \|\sum x_i \otimes e_{i,1}  \|_{C_n \otimes B(H)} 
	\leq \delta_p \cdot n.$$
	Then we can use Theorem \ref{thm1.11} for $E(x_i)$'s and the classical AGM inequality (here $\delta_p \leq \frac{1}{4} \gamma_p \leq \frac{1}{4d} \textless \frac{1}{3d} $).
	\begin{align*}
	E(P_{d}(x_{1},...,x_{n}))  &\leq P_{1}(Ex_{1},...,Ex_{n})\\
	&=\sum \frac {Ex_{i}}{n}=1.
	\end{align*}
	Using the upper bound above and Theorem \ref{th4.3}, we have the required inequality.
	\begin{equation*}
	\vertiii{P_{d}(x_{1},...,x_{n})}_{p_d}\leq \vertiii{P_{1}(Ex_{1},...,Ex_{n})}_{p_d}+\epsilon\leq 1+\epsilon.\qedhere
	\end{equation*} 
\end{proof}
~~\\

\subsection{ Application to Log concave measures }
~~\\

\mbox{} \hspace{4mm} In this section we want to study random AGM inequalities for log-concave measures.
\begin{defi}
	A Borel measure $\mu$ on $n$-dimensional Euclidean space $\mathbb{R}^{n}$ is called logarithmically concave (or log-concave) if for any compact subsets $A$ and $B$ of $\mathbb{R}^{n}$ and $0\leq\lambda\leq1$ we have
	$$\mu\big(\lambda A+(1-\lambda)B\big)\geq \mu\big(A\big)^{\lambda}\mu\big(B\big)^{(1-\lambda)}.$$
\end{defi}
\vspace{0.4in}
Let us recall the isotropic measure $\mu$ in $\mathbb{R}^{n}$. 
\begin{defi}
	The isotropic measure $\mu$ is the measure which satisfies  
	$$\int_{\mathbb{R}^{n}}|\langle\theta,x\rangle|^{2}d\mu(x)=L_{\mu}\|\theta\|^{2},$$ for all~~$\theta\in R^{n}$ where $L_{\mu}$ is denoted as isotropic constant.
\end{defi} 

Also let us recall Rosenthal's inequality, which will be used frequently in this section.
\begin{theorem}[Rosenthal inequality \cite{junge2013noncommutative}]Let $A_{i}$ be a fully independent sub-algebra over $N$ where $N\subset M$ and $M$ is a von Neumann algebra, and $1\leq p <\infty$. Let $x_{i}\in L_{p}(A_{i})$ with $E_{N}(x_{i})=0.$ Then 
	$$\|\sum_{i=1}^{n}x_{i}\|_{p}\leq C\max\{ \sqrt{p}\|\sum_{i=1}^{n} E_{N}(x_{i}^{*}x_{i}+x_{i}x_{i}^{*})^{1/2}\|_{p}, p (\sum_{i=1}^{n}\|x_{i}\|^{p}_{p})^{1/p} \} .$$
\end{theorem}

\begin{theorem}\label{th4.9} Let $n,~d\in \mathbb{N}$, $p\geq 2$. Let $(\mathbb{R}^{d},\mu)$ be log-concave Borel measure $\mu$ in isotropic position on $\mathbb{R}^{d}$ with constant $L$. Define random variable $y:\mathbb{R}^{d}\rightarrow \mathbb{R}^{d}$ by $y(\omega)=\frac{\omega}{\sqrt{L}}$ where $\omega\in \mathbb{R}^{d}$. Let $y_{i}$ be independent copies of $y$. Then $x_{i}(\omega):=|y_{i}(\omega)\rangle\langle y_{i}(\omega)|$ is  a $d\times d$ random matrix satisfying 
	\begin{enumerate}
		
		\item [(i)] $\forall i,$ $E x_{i}=1$, 
		
		\item [(ii)]$ \vertiii{\sum(x_i -Ex_i) }_{p} \leq \gamma_{p} \cdot n,$
	
		\item [(iii)]$\vertiii{ \sum x_{i}^{2}}^{1/2}_{p}\leq \gamma_{p}\cdot n,$\\
		where 
		$\gamma_{p} = \begin{cases}
		p^{1/2}\sqrt{\frac{d}{n}}+p^{5/2}\frac{d}{n} & p\geq \ln n~~ \text{or}\\
		2C\sqrt{\ln d}\delta^{1/2} & d\leq \frac{n}{\ln n^{5}} \\
		\mbox{~} 2C(\ln n)^{3}\delta & d\geq  \frac{n}{\ln n^{5}}\\
		
		\end{cases}$
	
	    \item [(iv)]Moreover, assume $\gamma_{p} \leq (1-\frac{2}{2+\varepsilon})\frac{1}{d}$, $\varepsilon\geq 0$, and $p_{k}:=\frac{p}{k}$. Then the following hold.
	    \begin{itemize}
	    	\item  
	 $\vertiii{P_{k}(x_{1},...,x_{n})-EP_{k}(x_{1},...,x_{n})}_{p_{k}}\leq \varepsilon.$
	 
	    \item The AGM inequality holds
	    $$\vertiii{P_{k}(x_{1},...,x_{n})}_{p_{k}}\leq (1+2\varepsilon).$$
	    \end{itemize}
	    \end{enumerate}
\end{theorem}

\begin{proof}
	We apply Rosenthal's inequality for $q\geq p$ to $x_{i}-1$ instead of $x_{i}$. 
	Let us introduce the norm in the space $L_{q}(S_{q})$ where $S_{q}$ is the Schatten class, $$|x|_{q}:=(E\|x_{i}\|^{q}_{S_{q}})^{\frac{1}{q}}=(\int\|x({\omega})\|^{q}_{q} d\mu)^{\frac{1}{q}}.$$ So, we have  
	\begin{align*}
	&\vertiii{ \sum(x_i -Ex_i) }_{q}\\
	&\leq  c \max \{ \sqrt{q} |\sum E((x_i-1)^*(x_i-1) + (x_i-1) (x_i-1)^*)|_{\frac{q}{2}}^{\frac{1}{2}}, q(\sum |x_i-Ex_i|_q^q)^{\frac{1}{q}} \}  \\
	&\leq   c\sqrt{q} |\sum E((x_i-1)^*(x_i-1) + (x_i-1) (x_i-1)^*)|_{\frac{q}{2}}^{\frac{1}{2}} + cq (\sum |x_i-Ex_i|_q^q)^{\frac{1}{q}}\\
	&\leq  2c\sqrt{q} |\sum E(x_i-1)^2 |_{\frac{q}{2}}^{\frac{1}{2}} +  cq (\sum |x_i-Ex_i|_q^q)^{\frac{1}{q}} \\
	&\leq  2c \sqrt{q}   | \sum Ex_i^2   |_{\frac{q}{2}}^{\frac{1}{2}}  +2 cqn^{1/q} |x_{1}|_{q} 
	\end{align*}
	By Rosenthal's inequality, we need to separately estimate the two terms of the right side. We denote
	\begin{equation}
	\text{I} = | (\sum Ex_i^2)^{1/2}) |_q  \text{ and }
	\text{II} = |x_{1}|_{q}. 
	\end{equation}
	We claim that $(ii)$ holds for $\gamma_{q}$ and $Ex_{i}^{2}\leq d Ex_{i}\leq cd\cdot1.$ Using Borel inequality (see \cite{milman1986asymptotic} where $\|.\|$ is seminorm), we have $$(E\|{y}\|^{q}_{X})^{\frac{1}{q}}\leq C_{q} E\|{y}\|_{X}\leq C_{q} (E\|{y}\|_{X}^{2})^{\frac{1}{2}}.$$ Recall that $E\|y\|^{2}=\sum\limits_{i=1}^{d}E|\langle e_{i},\frac{\omega}{\sqrt{L}}\rangle|^{2}=d.$~~So, we have for $x_{i}:=x_{1}=|y\rangle\langle y|$
	\begin{align*}
	\langle\theta, Ex_{1}^{2} \theta \rangle 
	&=E \langle \theta,y\rangle \langle y,y\rangle \langle \theta,y\rangle
	=E\|y\|^{2}|\langle\theta,y\rangle|^{2}\\
	&\leq (E\|y\|^{4})^{\frac{1}{2}}(E|\langle\theta,y\rangle|^{4})^{\frac{1}{2}}\\
	&\leq C_4^{4}~~ E\|y\|^{2} E(|\langle \theta,y\rangle|^{2})=C_{4}^{4}\cdot d~~\|\theta\|^{2}.
	\end{align*} 
	i.e. $Ex_{i}^{2}\leq d Ex_{i}\leq cd\cdot1$. This implies  $\| (\sum Ex_i^2)^{1/2}) \|_q  \leq C\cdot d^{1/2+1/q}$ 
	which proves our claim for (I). 
	For (II), note that the $q$-norm is defined to be $|x|_q= (E \tr|x|^q)^{\frac{1}{q}}.$
	Let's first take $q=m$ be an integer. We have
	\begin{align*}
	x_i^m =& |y_i \rangle \langle y_i |^m = y_i \rangle \langle y_i,y_i \rangle \cdots \langle y_i,y_i \rangle \langle y_i  \\
	=& y_i \rangle \|y_i \|^{2(m-1)} \langle y_i.
	\end{align*}
	Then, by using the Borel inequality (see \cite{milman1986asymptotic}), we have  
	\begin{align*}
	E \tr(x_i^m) =& E \tr(y_i \rangle \|y_i \|^{2(m-1)} \langle y_i) = E (\|y_i\|_{2}^{2m})\\
	&\leq (C\cdot 2m)^{2m}((E\|y\|_{2}^{2})^{1/2})^{2m}\\
	&\leq (C\cdot 2m)^{2m}d^{m}.
	\end{align*}
	So we get the inequality $|x|_{m} \leq (C\cdot 2m)^{2}d$ for arbitrary integer $m$. Then  for any real number $q$, we can find an integer $m$, such that $m\leq q \leq m+1$, and by interpolation between $m$ and $m+1$, we get 
	\begin{align}
	|x|_{q} \leq (C\cdot 2q)^{2}~d \label{eq4.4}. 
	\end{align} 
	Thanks to \eqref{eq4.4}, we can now prove condition (iii).
	\begin{align}\nonumber
	\vertiii{(\sum x_{i}^{2})^{1/2}}_q &\leq \vertiii{\sum x_{i}^{2}}^{1/2}_{\frac{q}{2}}
	\leq (\sum \vertiii{x_{i}^{2}}_{\frac{q}{2}})^{1/2}\\\nonumber
	&\leq (\sum \vertiii{x_{i}}_{q})^{1/2}=\sqrt{n}\vertiii{x_{1}}_{q}\\
    & \leq\sqrt{n} |x_{1}|_{q}\leq \sqrt{n} d~(C\cdot 2q)^{2}.
	\end{align}
    Combining (I) and (II) we obtain
	\begin{align*}
	\vertiii{\sum(x_i -1) }_{q} \leq  & \tilde{c}(qnd)^{1/2} d^{1/q}+ c q C n^{1/q}q^2d \\
	= & \tilde{C}(qn)^{1/2} d^{\frac{1}{2}+\frac{1}{q}} +  C' n^{1/q}q^3d. \\
	\end{align*}
	And then divide each term by $n$, we have 
	\begin{align*}
   \frac{\vertiii{\sum(x_i -Ex_i)}_{q}}{n}	
	&\leq C(q,d,n):= (\frac{q}{n})^{1/2} d^{\frac{1}{2}+\frac{1}{q}}+n^{\frac{1}{q}-1} q^{3}d\\
	&=(\frac{q}{n})^{1/2}~d^{\frac{1}{2}+ \frac{1}{q}}+\frac{qd}{n}q^{2}n^{\frac{1}{q}}\\
	&= d^{1/q} \Big(q^{1/2}(\frac{d}{n})^{1/2}+ q^3\frac{d}{n}~\frac{n^{1/q}}{d^{1/q}}\Big)\\
	&= d^{1/q}\Big((q)^{1/2} (\frac{d}{n})^{1/2}+ q^3 (\frac{d}{n})^{1-1/q}\Big).
	\end{align*}
	If we denote $\frac{d}{n}=\delta$, then
	\begin{align*}
	\frac{\vertiii{\sum(x_i -Ex_i)}_{q}}{n}&\leq d^{\frac{1}{q}}(q^{\frac{1}{2}}\delta^{\frac{1}{2}}+ q^3\delta^{1-\frac{1}{q}})\\
	&=d^{1/q}q^{1/2}\delta^{1/2}(1+q^{5/2}\delta^{1/2-1/q}).
    \end{align*}
    Now our goal is to find $\hat{\gamma}_{q} =\inf\limits_{q\geq q_{0}} {d^{1/q}q^{1/2}\delta^{1/2}(1+q^{5/2}\delta^{1/2-1/q})}$ by optimization over $q$ where $q_{0}\geq 2$. Define $f(q,\delta):= q^{5/2}\delta^{1/2-1/q}$ and consider $$g:=\ln f(q,\delta)=\frac{5}{2}\ln q+(\frac{1}{2}-\frac{1}{q})\ln \delta, $$ with derivative ${g}'=\frac{5}{2}\frac{1}{q}+\frac{1}{q^{2}}\ln \delta$. The critical point for $f(q,\delta)$ is $q(\delta)=\frac{2}{5} \ln\frac{1}{\delta}.$
    Since $f(q,\delta)$ is a convex function then it has no more than one minimum point which is $q(\delta)$. Then we have to consider the following cases for the choices of $q$,
    \begin{enumerate}
    \item $q_{0}\leq \ln d \leq q_{1}$ where $q_{1}=(\frac{1}{\delta})^{1/5}$.
    \item $\ln d \leq q_{0}\leq \ln n$ 
    \item $\ln d \leq \ln n\leq q_{0}$ 
    \end{enumerate}
   This can be done by using optimization over $q$ for the term $d^{1/q}q^{1/2}\delta^{1/2}$. For the first case, we choose $q=\ln d$ and $C(q,\delta)=2C\sqrt{\ln d}\delta^{1/2}$ where $f(q,\delta)\leq 1$. We also calculate $q_{1}$ which represents the upper bound for our choice of $q$ from $q^{5/2}\delta^{1/2}=1$. For the second case, if $(\frac{n}{d})^{1/5}\geq \ln\frac{n}{d}$, then we simply choose $q=\ln n$. This leads to $\frac{d}{n}\backsimeq \frac{1}{\ln n^{4}}\leq \frac{1}{\ln n^{5}}$.  
    We can summarize the cases in the following 
    $$\hat{\gamma}_{q} = \begin{cases}
    q^{1/2}\sqrt{\frac{d}{n}}+q^{5/2}\frac{d}{n} & q\geq \ln n ~~~\text{or}\\
    	2C\sqrt{\ln d}\delta^{1/2} & d\leq \frac{n}{\ln n^{5}} \\
    	\mbox{~} 2C(\ln n)^{3}\delta & d\geq  \frac{n}{\ln n^{5}}
    \end{cases}$$
    We apply the estimate for $q \geq p$ and appeal to Theorem \ref {th4.3} and Corollary \ref {col2.4} to deduce the AGM inequality.
		\end{proof}
\subsection{Wishart random variable matrices}

Let us recall the definition of Wishart random matrices. Let $[g^{i}_{r,s}]$ is a family of $d\times m$ Gaussian random matrices such that $i\in [1,n]$,~$r\in [1,d]$ and $s\in [1,m] $. Define $G_{i}=\frac{1}{\sqrt{m}}[g^{i}_{rs}]$ and $x_{i}=G_{i}G_{i}^{*}$. We call the matrices $x_i$ $d\times d$ Wishart random matrices. Then we have  $Ex_{i}=EG_{i}G_{i}^{*}=1$ which implies that $\sum_{i=1}^{n}Ex_{i}=n$. In this section we assume that $m\geq n$. Let us list some useful lemmas which will be used in the main theorem. Each of these lemmas proves one of the conditions of Theorem \ref{th4.3} separately.
 

\begin{lemma} \label{lemma51}
	Let $\varepsilon_{q,m,n,d}=  \Big(\frac{\sqrt{d}+ \sqrt{m}}{\sqrt{m}}\Big)^{2} \frac{q}{\sqrt{n}}$. Those 
	$d \times d $ Wishart random matrices $\{x_i\}$ from above satisfy 
	\begin{equation} 
	\frac{1}{n}\vertiii{ (\sum x_i^2)^{1/2}}_q \leq \varepsilon_{q,m,n,d} .
	\end{equation}
\end{lemma}
\begin{proof}
	Denote $A= \frac{1}{\sqrt{m}} \sum\limits_{r,s} g_{r,s} e_{r,s}.$
	Then for all $h \in H,$~and $x=AA^{*}$
	\begin{align*}
	E(h, x^2h)= & E(h, |AA^*|^2h) =  E(h, AA^*AA^* h)
	= E(AA^*h, AA^*h) \\
	= & E\|AA^*h\|^2 
	\leq  E( \|A\|_{op}^2 \cdot \|A^{*}h\|^2) 
	\leq  E \|A\|_{op}^2 \cdot E\|A^*h\|^2. \\
	\end{align*}
	Note that $E\|A^*h\|^2= E(h,  A^*Ah) =\|h\|^2$.
	Using Chevet's inequality \cite {gordon1985some},
	\begin{equation} 
	E\|A\|=E\|\sum\limits_{r=1}^{d}\sum\limits_{s=1}^{m} g_{r,s}e_{r}{\otimes} e_{s} \|_{X \check{\otimes} Y} \leq  E(\|\sum\limits_{s=1}^{m}g_{r,s}e_{s}\|)+ E(\|\sum\limits_{r=1}^{d}g_{r,s}e_{r}\|),
	\end{equation} 
	where $X=l_{2}^{m}$ and $Y=l_{2}^{d}$.
	We deduce that if $A=\frac{1}{\sqrt{m}}\sum g^{i}_{rs} e_{r}\otimes e_{s}$ then by using Kahane's inequality we have that  
	\begin{equation}\label{iq4.7}
	(E \|A\|_{op}^2)^{1/2} \leq \sqrt{2}\Big(\frac{\sqrt{d}+ \sqrt{m}}{\sqrt{m}}\Big) =:C(d,m)
	\end{equation}
	Therefore 
	$$\vertiii{x_i}_2=(E\|x_i\|_{op}^2)^{1/2}\leq C(d,m). $$
	For $q\geq 2$
	\begin{align} 
	\vertiii{(\sum x_i^2)^{1/2}}_q 
	&= \vertiii{\sum x_i^2}_{q/2}^{1/2}  \nonumber
	\leq (\sum \vertiii{x_i^2}_{q/2})^{1/2}   \\\nonumber
	&\leq  (\sum \vertiii{x_i}_q^2)^{1/2}  \leq \sqrt{n} \vertiii{x_i}_q=\sqrt{n} [(E\|A\|^{2q})^{1/2q}]^{2}\\\label{iq.4.8}
	& \leq \sqrt{n} (\sqrt q)^{2} [(E\|A\|^{2})^{1/2}]^{2}= 2q\sqrt{n}(\frac{\sqrt{d}+ \sqrt{m}}{\sqrt{m}})^2 
	\end{align} 
	The last inequality comes from Kahane's inequality (see proposition 3.3.1 and proposition 3.4.1 in \cite{kwapien1992random}) and inequality \eqref{iq4.7}.
	Thus, taking $\varepsilon_{q,m,n,d}=  \big(\frac{\sqrt{d}+ \sqrt{m}}{\sqrt{m}}\big)^2 \frac{2q}{\sqrt{n}}$, we have
	\begin{equation*}
     \frac{1}{n}\vertiii{ (\sum x_i^2)^{1/2}}_q \leq \varepsilon_{q,m,n,d} .\qedhere
	\end{equation*}
\end{proof}

The following lemma is used to prove the first condition in Theorem \ref{th4.3}.
\begin{lemma}\label{lemma5}
	For $d\times d$ Wishart random variables $x_i$, the following is satisfied 
	\begin{equation}
	\frac{1}{n}\vertiii{\sum (x_i -E x_i)  }_q\leq \gamma^{'} _q , 
	\end{equation}
	where $\gamma^{'} _q = \begin{cases}
	C'\ln d~ \sqrt{\frac{\ln d}{n}}                              & q \leq \ln d\leq n \\
	\mbox{~} C'd^{\frac{1}{q}}q \max\{\sqrt{\frac{q}{n}},\frac{q}{n} \} & q \geq \ln d~~.
	\end{cases}$
\end{lemma}
\begin{proof}
By Rosenthal's inequality, we have 
	\begin{align*}
	&\vertiii{\sum (x_i -E x_i )  }_q\\
	&\leq (E\|\sum\limits_{i}(x_{i}-Ex_{i})\|^{q}_{q})^{\frac{1}{q}}\\
	&\leq_{c} \sqrt{q}  (E| (\sum_i Ex_i^2 )^{1/2}   |_q^q)^{1/q} +q (\sum |x_i -Ex_i|_q^q)^{1/q} \\
	& \leq_c \sqrt{q}  (E | (\sum_i Ex_i^2 )^{1/2} |_q^q)^{1/q} +  qn^{\frac{1}{q}} \cdot \max_i  (E|x_i|_q^q)^{\frac{1}{q}}
	\\
	& \leq \sqrt{q} d^{\frac{1}{q}} (E|  (\sum x_i^2)^{\frac{1}{2}}|_{\infty}^q)^{1/q} + qn^{\frac{1}{q}} d^{\frac{1}{q}} q[\frac{\sqrt{d}+\sqrt{m}}{\sqrt{m}}]^2\\
	& \leq \sqrt{qn} d^{\frac{1}{q}} q[1+\sqrt{\frac{d}{m}}]^2+qn^{\frac{1}{q}} d^{\frac{1}{q}} q[1+\sqrt{\frac{d}{m}}]^2	 \\
	& \leq d^{\frac{1}{q}}[1+\sqrt{\frac{d}{m}}]^2(\sqrt{qn}q + q^2n^{\frac{1}{q}}).
	\end{align*}
	The second-to-last inequality uses Kahane's inequality \cite{kwapien1992random} and inequality \eqref{iq.4.8}. Dividing the inequality by $n$, we obtain
	\begin{equation} \label{lem413}
	\frac{\vertiii{\sum (x_i -E x_i )  }_q}{n}   \leq d^{\frac{1}{q}}[1+\sqrt{\frac{d}{m}}]^2q(\sqrt{\frac{q}{n}} + \frac{\sqrt{q}}{n}\sqrt{q}n^{\frac{1}{q}})~.
	\end{equation}
	Let $2\leq q_{0}\leq q$. We have two cases to estimate the upper bound:
	\begin{enumerate}
	\item $q_{0} \leq \ln d\leq n$ \\
	\item $\ln d\leq q_{0}\leq q$
	\end{enumerate}
	We follow the optimization for $q$ from the proof of Theorem \ref{th4.9}. Define $f(q)=\sqrt{q}n^{\frac{1}{q}}$, and consider $g(q) = \ln f(q)=\frac{1}{2} \ln q + \frac{1}{q}\ln n$, then 
	$g'(q)= \frac{1}{2q}-\frac{\ln n}{q^2}=0$ at $ q= 2~{\ln n }$. Then 
 
	$$\frac{\sqrt{q}}{n}f(q) \leq \begin{cases}
	C\sqrt{\frac{q}{n}} & 2\leq q \textless n \\
	\mbox{~}C\frac{q}{n} & q \geq n 
	\end{cases}$$
	Moreover, by \eqref{lem413}, when $d\leq m$, we obtain
	\begin{align*}
	d^{\frac{1}{q}}[1+\sqrt{\frac{d}{m}}]^2q(\sqrt{\frac{q}{n}} +\sqrt{ \frac{q}{n}}\sqrt{q}n^{\frac{1}{q}-\frac{1}{2}}) & \leq 2Cd^{\frac{1}{q}}[1+\sqrt{\frac{d}{m}}]^2q \max\{\sqrt{\frac{q}{n}} , \frac{q}{n} \} \\
	& \leq 8Cd^{\frac{1}{q}}q \max\{\sqrt{\frac{q}{n}} , \frac{q}{n} \}. 
	\end{align*} 
	Denote  $F(d,n)= 8Cd^{\frac{1}{q}}q \max\{\sqrt{\frac{q}{n}} , \frac{q}{n} \} $. 	
	We choose $q=\ln d$ and we get that 
	
	$F(d,n)=C'\ln d~\sqrt{\frac{\ln d}{n}} $ if we have $q_{0} \leq \ln d\leq n$. Otherwise we choose $q\geq q_{0}$, and we get $F(d,n)= C'd^{\frac{1}{q}}q \max\{\sqrt{\frac{q}{n}},\frac{q}{n} \}$. 
    Moreover, $$\hat{\gamma}_{q}=\begin{cases}
     C'\ln d~ \sqrt{\frac{\ln d}{n}}                              & q \leq \ln d\leq n \\
    \mbox{~} C'd^{\frac{1}{q}}q \max\{\sqrt{\frac{q}{n}},\frac{q}{n} \} & q \geq \ln d~~.
    \end{cases}$$
     
     We apply the estimate for $q\geq p$ and appeal to  Theorem \ref {th4.3} and Corollary \ref {col2.4}.\qedhere
	
	
	
	
\end{proof}

Now, we can prove the AGM inequality for random matrices which holds up to $(1+\varepsilon).$

\begin{theorem}
	Let $\{x_{i}\}$ be a family of self-adjoint family of  $d\times d$ Wishart random matrices. For $ 2\leq p\leq \ln d \leq  n $, we have 
	
	\begin{enumerate}
		\item[(i)] $\vertiii{\sum\limits_{i} (x_{i}-E(x_{i}))}_{p}\leq \gamma_{p}n;$
		
		\item[(ii)] $\frac{1}{n}\sum\limits_{i=1}^{n}E(x_{i})=1;$
		
		\item [(iii)]$\vertiii{(\sum\limits_{i} x^{2}_{i})^{\frac{1}{2}}}_{p}\leq \gamma_{p}n, $ where $\gamma_{p}=
	C'\ln d~ \sqrt{\frac{\ln d}{n}},~~ p_{0} \leq \ln d\leq n; \\
	$

		\item [(iv)]  
		 Moreover, for $\varepsilon \geq 0$ if $\gamma_{p} \leq \frac{\varepsilon}{3k}$, $p_{k}:=\frac{p}{k}$ then the following hold.
		 \begin{itemize} 
		 	\item $\vertiii{P_{k}(x_{1},...,x_{n})-EP_{k}(x_{1},...,x_{n})}_{p_{k}}\leq \varepsilon.$
		\item The random AGM inequality holds, 
			    $$\vertiii{P_{k}(x_{1},...,x_{n})}_{p_{k}}\leq (1+2\varepsilon).$$
		\end{itemize}
		\end{enumerate}

\end{theorem} 
\begin{proof}
	Condition $(ii)$ comes from definition of the Wishart random matrices. For condition $(i)$ we directly use Lemma \ref{lemma5} for the case when $p_{k}\leq \ln d \leq n$. For condition $(iii)$, we use Lemma \ref{lemma51}. This implies that all the conditions of Theorem \ref{col2.4} are satisfied, since $p_k \leq \ln d \leq n $. Thus, we get the random AGM inequality.
\end{proof}

\section{application on Pisier's construction and freely independent}
Let $(M,\tau)$ be a von Neumann algebra where $\tau$ is a faithful normal and normalized trace. An example of a finite von Neumann algebra is given by the group von Neumann algebra $L(G)$ associated to the left regular representation $\lambda(G)$ of a discrete group $G$. It is defined as the strong operator closure of the linear span of $\lambda(G)$.  Recall that $L_{p}(M,\tau)$ where $1\leq p <\infty$ is defined as the completion of $M$ with respect to the norm $\|x\|_{p}=(\tau(|x|^{p}))^{1/p}$ (see Pisier \cite{pisier2003non} for more details). Note that $L(G)=L_{\infty}(L(G))$ and $L(G)\subset L_{p}(L(G))$. We want to prove a version of the AGM inequality with respect to the norm $\|.\|_{p}$. For this version of the AGM inequality, we need the following key lemma. 
\begin{lemma}\label{l4}
Let $M$ be a von Neumann algebra. Let $\nu$ be a partition. Then there exists a group $G$ and $b_{i}(j)\in L(G)$ such that for $x_{i}(j)\in L_{p}(M)$,  the elements $X_{i}(j)=b_{i}(j) \otimes x_{i}(j) \in L_{p}(L(G)\otimes M )$ satisfy 
\begin{equation}
[\nu]=E_{M}\sum\limits_{i_1,i_2,i_3,\ldots ,i_{d}} X_{i_{1}}(1)X_{i_{2}}(2)...X_{i_{d}}(d).
\end{equation}
Moreover, 
\begin{align}\label{5}
&\|\sum\limits_{i}X_{i}(j)\|_{p}
\leq \\
&\begin{cases} \nonumber
C\max \Big\{\|(\sum x_{i}(j)^{*}x_{i}(j))^{1/2}\|_{p},\|(\sum x_{i}(j)x_{i}(j)^{*})^{1/2}\|_{p}\Big\}~& j\in A_{n.s}\in \sigma_{n.s}\\
~~\|\sum ~x_{i}(j)~\|_{p}         &\{j\}\in \sigma_{s},
\end{cases}
\end{align}
\end{lemma}
where $C$ is a universal constant. Note that $b_{i}(j)=1$ if $\{i\}\in \sigma_{s}$.
\begin{rem}\rm{
The norm inequality \eqref{5} was proved by Pisier for even integers $p\geq 2$ in \cite{pisier2003non}. The general case follows from \cite{junge2007rosenthal}.} 
\end{rem}
Now we can state the AGM inequality for $L_{p}(M)$ where $p\geq d.$
\begin{theorem}
Let $M$ be a von Neumann algebra and $x_{i}\in L_{p}(M,\tau)_{s.a}$ satisfy the following condition for some $\delta\geq 0$, $$\|(\sum x_{i}^{2})^{1/2}\|_{p}\leq \delta \|\sum x_{i}\|_{p}.$$ Then we have
\begin{equation}
\|P_{d}(x_{1},...,x_{n})\|_{\frac{p}{d}}\leq \Big(1+(\delta C)(d!-1)\Big)\frac{n^{d}(n-d)!}{n!}\|\frac{1}{n}\sum\limits_{1}^{n} x_{i}\|_{p}^{d}.
\end{equation}

\end{theorem}
We will only give the sketch of the proof of this theorem since it is similar to the proof of Theorem \ref{w1} for $p_{d}=\frac{p}{d}\geq 1$.
\begin{proof}
By using Lemma \ref{l4}, H\"older's inequality and the contraction of conditional expectation we have that
\begin{align*}
\|\langle\sigma \rangle\|_{p_{d}}&\leq \|\sum x_{i}\|_{p}^{d} + \sum\limits_{\upsilon \gneqq \dot{0}} |\mu(\dot{0},\nu)| C^{|v_{n.s}|}\|(\sum x_{i}^{2})^{1/2}\|_{p}^{|v_{n.s}|} \|\sum x_{i}\|_{p}^{|v_{s}|}\\
&\leq \|\sum x_{i}\|_{p}^{d} + \sum\limits_{\upsilon \gneqq \dot{0}} |\mu(\dot{0},\nu)| (\delta C)^{|v_{n.s}|}\|\sum x_{i}\|_{p}^{d}.
\end{align*}
Thus for $\delta C\leq 1$
\begin{equation}
\|P_{d}(x_{1},...,x_{n})\|_{p_{d}}\leq (1+(\delta C)(d!-1))\frac{n^{d}(n-d)!}{n!}\|\frac{1}{n}\sum x_{i}\|_{p}^{d}.\qedhere
\end{equation}
\end{proof}
\begin{rem}
	If $\delta \leq 1 $, we get the AGM inequality with a constant $C(d,n)=C^{d}d^{d}$.
\end{rem}
\vspace{0.1in}
As a matter of completeness, we want to include the limit case of the Wishart random matrices as an application for the AGM inequality. Let's first give the definition of freely independent von Neumann algebra (for more details see \cite {voiculescu1992free}).
\begin{defi}
Let $\{A_{i}\}$ be a family of unital von Neumann subalgebras of $A$. Then $\{A_{i}\}$  is called a freely independent algebra (with respect to a unital linear functional $\phi$ ) if $\phi(x_{1}... x_{n})=0$ whenever $\phi(x_{j})=0$ for all $x_{j}\in A_{i_{j}}$ and $i_{1}\neq i_{2},i_{2}\neq i_{3},...$
\end{defi}
We say that operators $x_{i}\in A_{i}$ are freely independent if their algebra $\{A_{i}\}$ are freely independent. In the following theorem we prove the deviation inequality up to $\varepsilon$ and apply this to the AGM inequality.
\begin{theorem}
If $\{x_{i}\}$ are freely independent in von Neumann algebra $M$ such that
\begin{enumerate}
\item $E_{M}(x_{i})=1$
\item $\sup\limits_{i}\|x_{i}\|\leq C$ and $2+(4\sqrt{n})C\leq \frac{\varepsilon n}{3d},$
\end{enumerate}
then 
\begin{enumerate}
\item $\|P_{d}(x_{1},...,x_{n})-EP_{d}(x_{1},...,x_{n})\|_{\infty}\leq \varepsilon$
\item $\|P_{d}(x_{1},...,x_{n})\|_{\infty}\leq  1+ \varepsilon.$
\end{enumerate}
\end{theorem}
\begin{proof}
Let $a_{i}=x_{i}-1$. By assumption we have $E_{M}(a_{i})=0$. By simple modification of Voiculescu's inequality \cite{junge2005embedding} , we get that 
\begin{align}\nonumber
\|(\sum a_{i}^{2})^{1/2}\|=\|\sum e_{i1}\otimes a_{i}\|&\leq \sup \|a_{i}\|+2\|(\sum E_{M}(a_{i}^{2}))^{1/2}\| \\\nonumber
&\leq 2(1+C)+2\sqrt{n}C\nonumber\\
&\leq 2+(4\sqrt{n}C)\leq \frac{\varepsilon n}{3d}.
\end{align}
Indeed, $\|a_{i}\|=\|x_{i}-1\|\leq 1+\|x_{i}\|\leq 1+C$ and 
\begin{align*}
E_{M}(a_{i}^{2})&=E_{M}(x_{i}-E_{M}(x_{i}))^{2}=E_{M}(x_{i}^{2})-E_{M}(x_{i})^{2}\\
&\leq E_{M}(x_{i}^{2})=E_{M}(x_{i}^{1/2}|x_{i}|x_{i}^{1/2})\\
&\leq \|x_{i}\|E_{M}(x_{i})=\|x_{i}\|\leq C.
\end{align*} 
Again, using Voiculescu's inequality we have, 
\begin{align}
&\|\sum a_{i}\|\leq \sup \|a_{i}\|+2\|(\sum E_{M}(a_{i}^{2}))^{1/2}\|\leq 2(1+C)+2\sqrt{n}C\leq \frac{\varepsilon n}{3d}.
\end{align}
Following the proof of Proposition \ref{lm4.1}, we get 

\begin{equation}
\|[\nu]\|\leq \|(\sum a_{i}^{2})^{1/2}\|^{d-|v_{s}|} \|\sum a_{i}\|^{|v_{s}|}\leq (\frac{\varepsilon n}{3d})^{d}.
\end{equation}
Applying the techniques of Proposition \ref{lm4.1} to the case $p=\infty$, we have 
\begin{equation}
\|P_{k}(a_{1},...,a_{n})-EP_{k}(a_{1},...,a_{n})\|\leq 2\frac{(n-k)!}{n!} k!(\frac{\varepsilon n}{3d})^{k}.
\end{equation}
Then we have
\begin{align}\nonumber
\|P_{d}(x_{1},...,x_{n})-EP_{d}(x_{1},...,x_{n})\|&=\|\sum_{k=1}^{d}{{d}\choose{k}}(P_{k}(a_{1},...,a_{n})-EP_{k}(a_{1},...,a_{n}))\|\\
&\leq 2\sum\limits_{k=1}^{d}\underbrace{\frac{d!(n-k)!}{(d-k)!n!}n^{k}}_{\text{$f(n)$ is a decreasing function}}(\frac{\varepsilon}{3d})^{k}\\
&\leq 2\sum\limits_{k=1}^{d}(d)^{k}(\frac{\varepsilon}{3d})^{k}
\leq \varepsilon.\label{5.6}\nonumber
\end{align}
Then we have to apply Theorem \ref{thm1.11} for $y_{i}=Ex_{i}$ instead of $x_{i}$, where $\frac{\sum y_{i}}{n}=1$. Note that by free independence, we have $EP_{d}(x_{1},...,x_{n})=P_{d}(Ex_{1},...,Ex_{n})$ using the fact that $\{x_{n}\}$ in $P_{d}(x_{1},...,x_{n})$ has no repetition.
\begin{equation*}
\|P_{d}(x_{1},...,x_{n})\|\leq \|P_{1}(Ex_{1},...,Ex_{n})\|+\epsilon\leq 1+\epsilon.\qedhere
\end{equation*}  
\end{proof}
\begin{rem}\rm{
The norm version of the AGM inequality also holds for the family of freely independent $\{x_{i}\}$.
Indeed, we have that 
\begin{equation}
\|P_{d}(x_{1},...,x_{n})\|\leq (1+\tilde{\varepsilon}) \|\frac{1}{n}\sum x_{i}\|^{d}.
\end{equation}
In this case we use again the Voiculescu inequality and deduce that  
$$\|\frac{1}{n}\sum x_{i}-\frac{1}{n}\sum Ex_{i}\|\leq \frac{\varepsilon}{3d}.$$
This implies $\|\frac{1}{n}\sum x_{i}\|\geq 1-\frac{\varepsilon}{3d}.$
Hence, $$\|P_{d}(x_{1},...,x_{n})\|^{1/d}\leq \frac{(1+\varepsilon)^{1/d}}{(1-\frac{\varepsilon}{3d})}\|\frac{1}{n}\sum x_{i}\|.$$
Note that $\frac{(1+\varepsilon)^{1/d}}{(1-\frac{\varepsilon}{3d})}\approx \frac{1+\varepsilon}{1-\varepsilon}$. This is true  since we have $$(1-x)^{n}\geq (1-nx)$$ for $x\in [0,1]$ and $n\geq 1$.
Applying this inequality for $x=\frac{\varepsilon}{3d}$, we have $$(1-\frac{\varepsilon}{3d})^{d}\geq (1-\frac{\varepsilon}{3d}d)=(1-\frac{\varepsilon}{3}).$$ Thus, AGM inequality is true up to the constant $\frac{1+\varepsilon}{1-\frac{\varepsilon}{3}}\approx 1+\tilde{\varepsilon}.$} 
\end{rem}

Another interesting application for freely independent copies $\{x_{i}\}$ is given as follows:
\begin{cor}
Let $\{x_{i}\}$ be self-adjoint freely independent copies over an algebra $B$ such that $E_{B}(x_{1})=1_{B}$ and $\|x_{1}\|\leq C$. Then the AGM inequality holds up to $(1+\varepsilon)$.
\end{cor}
\begin{proof}
Using the free independence for the $\{x_{i}\}$'s, where $d\leq p\leq \infty$ we get   
\begin{enumerate}
\item  $E_{B}(x_{i})=1_{B}$
\item $\|x_{i}\|_{p}=\|x_{1}\|_{p}\leq C$
\item $\|(\sum x_{i}^{2})^{1/2}\|_{p}\leq \tilde{c}\|x_{i}\|_{p}n^{1/p}+\sqrt{n}\|(E_{M}x_{i}^{2})^{1/2}\|_{p}$
\end{enumerate}
Indeed, for the property (3) we just apply a version of Voiculescu's inequality  for free variables \cite{junge2007rosenthal}, 
\begin{align*}
&\|\sum x_{i}\otimes e_{i1}\|_{p}\leq c (\sum\|x_{i}\|^{p})^{1/p} + \|(\sum E_{M}(x_{i}^{*}x_{i}))^{1/2}\|_{p}\\
&\leq \tilde{C} n^{1/p}+\sqrt{n}\|(E_{M}x_{1}^{2})^{1/2}\|_{p}
\end{align*}
Note that 
\begin{align*}
&\|\sum\limits_{1}^{n} x_{i}\|_{p}\geq \|\sum\limits_{1}^{n}E_{M}(x_{i})\|_{p}-\|\sum\limits_{1}^{n} \Big(x_{i}-E_{M}(x_{i})\Big)\|_{p}\\
&\geq n\|E_{M}(x_{1})\|_{p}-\underbrace{\Big (\tilde{C}n^{1/p}+\sqrt{n} \|E_{M}(x_{i}^{2})^{1/2}\|_{p}\Big)}_\text{A}
\end{align*}
Now, if $ A\leq \frac{n}{2}\|E_{M}(x_{1})\|$, then we have 
\begin{align*}
\|(\sum\limits_{1}^{n} x_{i}^{2})^{1/2}\|_{p}&\leq2 \frac{\tilde{C}n^{1/p}+\sqrt{n}\|E_{M}(x_{1}^{2})^{1/2}\|_{p}}{n\|E_{M}(x_{1})\|_{p}} \|\sum\limits_{1}^{n} x_{i}\|_{p}\\
&= n^{-1/2}\underbrace{\Bigg(\frac{2\tilde{C}n^{1/p-1/2}+2\|E_{M}(x_{1}^{2})^{1/2}\|_{p}}{\|E_{M}(x_{1})\|_{p}}\Bigg)}_{C_{n}} \|\sum\limits_{1}^{n} x_{i}\|_{p}.
\end{align*}
Then we get 
\begin{equation*}
\|(\sum\limits_{1}^{n} x_{i}^{2})^{1/2}\|_{p}\leq \delta_{n} \|\sum\limits_{1}^{n}x_{i}\|_{p},
\end{equation*}
where $\delta_{n}=\frac{C_{n}}{\sqrt{n}}$. Then for $\sqrt{n}\gg d!$ we have $\delta_{n} \rightarrow 0$.
This implies that when $n$ is large enough we get the AGM inequality as follows:
\begin{equation}
\|P_{d}(x_{1},...,x_{n})\|_{\frac{p}{d}}\leq (1+\varepsilon) \|\sum\limits_{1}^{n} x_{i}\|^{d}_{p}.\qedhere
\end{equation}
\end{proof}

\mbox{}
\nocite{*}
\bibliographystyle{plain}
\bibliography{agm}

\begin{thebibliography}{10}

\bibitem{andrews1998theory}
George~E Andrews.
\newblock {\em The theory of partitions}, volume~2.
\newblock Cambridge university press, 1998.

\bibitem{bottou1998online}
L{\'e}on Bottou.
\newblock Online learning and stochastic approximations.
\newblock {\em On-line learning in neural networks}, 17(9):25, 1998.

\bibitem{feng2012towards}
Xixuan Feng, Arun Kumar, Benjamin Recht, and Christopher R{\'e}.
\newblock Towards a unified architecture for in-rdbms analytics.
\newblock In {\em Proceedings of the 2012 ACM SIGMOD International Conference
  on Management of Data}, pages 325--336. ACM, 2012.

\bibitem{gordon1985some}
Yehoram Gordon.
\newblock Some inequalities for {G}aussian processes and applications.
\newblock {\em Israel Journal of Mathematics}, 50(4):265--289, 1985.

\bibitem{hardy1952inequalities}
Godfrey~Harold Hardy, John~Edensor Littlewood, and George P{\'o}lya.
\newblock {\em Inequalities}.
\newblock Cambridge university press, 1952.

\bibitem{junge2002doob}
Marius Junge.
\newblock Doob's inequality for non-commutative martingales.
\newblock {\em Journal fur die Reine und Angewandte Mathematik}, pages
  149--190, 2002.

\bibitem{junge2005embedding}
Marius Junge.
\newblock Embedding of the operator space oh and the logarithmic ‘little
  {G}rothendieck inequality’.
\newblock {\em Inventiones mathematicae}, 161(2):225--286, 2005.

\bibitem{junge2007rosenthal}
Marius Junge, Javier Parcet, Quanhua Xu, et~al.
\newblock Rosenthal type inequalities for free chaos.
\newblock {\em The Annals of Probability}, 35(4):1374--1437, 2007.

\bibitem{junge2003noncommutative}
Marius Junge and Quanhua Xu.
\newblock Noncommutative {B}urkholder/{R}osenthal inequalities.
\newblock {\em Annals of probability}, pages 948--995, 2003.

\bibitem{junge2013noncommutative}
Marius Junge, Qiang Zeng, et~al.
\newblock Noncommutative {B}ennett and {R}osenthal inequalities.
\newblock {\em The Annals of Probability}, 41(6):4287--4316, 2013.

\bibitem{kwapien1992random}
Stanislaw Kwapien and Wojbor~A Woyczynski.
\newblock {\em Random series and stochastic integrals: single and multiple}.
\newblock Boston, 1992.

\bibitem{lust1986inegalites}
Fran{\c{c}}oise Lust-Piquard.
\newblock In{\'e}galit{\'e}s de khintchine dans cp (1< p<∞).
\newblock {\em CR Acad. Sci. Paris}, 303:289--292, 1986.

\bibitem{lust1991non}
Fran{\c{c}}oise Lust-Piquard and Gilles Pisier.
\newblock Non commutative {K}hintchine and {P}aley inequalities.
\newblock {\em Arkiv f{\"o}r Matematik}, 29(1):241--260, 1991.

\bibitem{milman1986asymptotic}
Vitali~D Milman and Gideon Schechtman.
\newblock Asymptotic theory of finite dimensional normed spaces.
\newblock 1986.

\bibitem{nesterov2012efficiency}
Yu~Nesterov.
\newblock Efficiency of coordinate descent methods on huge-scale optimization
  problems.
\newblock {\em SIAM Journal on Optimization}, 22(2):341--362, 2012.

\bibitem{paulsen2002completely}
Vern Paulsen.
\newblock {\em Completely bounded maps and operator algebras}, volume~78.
\newblock Cambridge University Press, 2002.

\bibitem{pisier2000inequality}
Gilles Pisier et~al.
\newblock An inequality for $p$-orthogonal sums in non-commutative $ l\_
  $\{$p$\}$ $.
\newblock {\em Illinois Journal of Mathematics}, 44(4):901--923, 2000.

\bibitem{pisier2003non}
Gilles Pisier and Quanhua Xu.
\newblock Non-commutative lp-spaces.
\newblock {\em Handbook of the geometry of Banach spaces}, 2:1459--1517, 2003.

\bibitem{randrianantoanina2002non}
Narcisse Randrianantoanina.
\newblock Non-commutative martingale transforms.
\newblock {\em Journal of Functional Analysis}, 194(1):181--212, 2002.

\bibitem{recht2011parallel}
Benjamin Recht and Christopher Re.
\newblock Parallel stochastic gradient algorithms for large-scale matrix
  completion. submitted for publication.
\newblock {\em Preprint available at http://pages. cs. wisc. edu/\~{}
  brecht/publications. html}, 2011.

\bibitem{recht2012beneath}
Benjamin Recht and Christopher R{\'e}.
\newblock Beneath the valley of the noncommutative arithmetic-geometric mean
  inequality: conjectures, case-studies, and consequences.
\newblock {\em arXiv preprint arXiv:1202.4184}, 2012.

\bibitem{rota1964foundations}
Gian-Carlo Rota.
\newblock On the foundations of combinatorial theory i. theory of m{\"o}bius
  functions.
\newblock {\em Probability theory and related fields}, 2(4):340--368, 1964.

\bibitem{voiculescu1992free}
Dan~V Voiculescu, Ken~J Dykema, and Alexandru Nica.
\newblock {\em Free random variables}.
\newblock Number~1. American Mathematical Soc., 1992.

\end{thebibliography}
\end{document}